\documentclass[12pt]{article}
\usepackage{amsmath,amssymb,amsthm}
\usepackage[left=1in,right=1in,top=1in,bottom=1in]{geometry}
\usepackage{fouriernc,graphicx}

\newtheorem{theorem}{Theorem}
\newtheorem{lemma}[theorem]{Lemma}
\newtheorem{prop}[theorem]{Proposition}
\newtheorem{cor}[theorem]{Corollary}
\newtheorem{deF}[theorem]{Definition}

\newtheorem{alg}[theorem]{Algorithm}

\newtheorem{example}[theorem]{Example}

\newtheorem{conjecture}[theorem]{Conjecture}

 \renewcommand{\implies}{\Rightarrow}

\newcommand{\zz}{\mathbb{Z}}
\newcommand{\nat}{\mathbb{N}}
 
\newcommand{\real}{\mathbb{R}}

\newcommand{\ncom}{\newcommand}

\renewcommand{\implies}{\Rightarrow}

\def\abs#1{\left\vert #1 \right\vert}

\def\set#1{\lbrace #1 \rbrace}

\newcommand{\res}{\operatorname{res}}

\newcommand{\sig}{\sigma}

\newcommand{\al}{\alpha}
\newcommand{\ome}{\omega}
\newcommand{\ep}{\epsilon}
 \ncom{\Del}{\Delta}

\newcommand{\zec}{Zeckendorf}

\newcommand{\Lunique}{unique $\cL$-representation property}
\newcommand{\Eunique}{unique $\cE$-representation property}

\newcommand{\fib}{Fibonacci}

\newcommand{\cL}{\mathcal{L}}

\newcommand{\clz}{$\cL$-\zec}
\newcommand{\Lz}{\clz}

\newcommand{\aclz}{ascending $\cL$-\zec}

\newcommand{\cE}{\mathcal{E}}
\newcommand{\Ez}{$\cE$-\zec}
\newcommand{\dver}{descending version}

\newcommand{\lex}{lexicographical}
\newcommand{\dord}{<_\text{d}}

\newcommand{\aord}{<_\text{a}}
\newcommand{\aordeq}{\le_\text{a}}
\newcommand{\aorcoll}{ascendingly-ordered}
\newcommand{\dorcoll}{descendingly-ordered}
\newcommand{\cf}{coefficient function}
\newcommand{\UI}{\mathbf{I}}
\newcommand{\OI}{\UI}
\newcommand{\ord}{\mathrm{ord}}

\newcommand{\zquote}[1]{\lq\lq #1\rq\rq}

\newcommand{\seq}[1]{\set{#1_k}_{k=1}^\infty}
\newcommand{\wt}{\widetilde}
\newcommand{\zml}{\zec\ multiplicity list}
\newcommand{\bbeta}{\bar\beta}
\newcommand{\EEunique}{unique $\cE$-representation property}
\newcommand{\Eb}{$\hbeta$-block}
\newcommand{\Ebb}{$\bbeta$-block}
\newcommand{\dbl}{$\delta$-block}

\newcommand{\hbeta}{\hat \beta}

\newcommand{\immsucc}{immediate successor}
\newcommand{\immpred}{immediate predecessor}
\newcommand{\resab}[1]{\res_{#1}}
\newcommand{\td}{\tilde}
\newcommand{\reS}{\ \mathrm{res}\ }
\newcommand{\rem}{\mathrm{rem}}

\newcommand{\Lb}{$L$-block}
\newcommand{\LLb}{\Eb}

\begin{document}

\begin{center}\bf
\Large The Weak Converse of Zeckendorf's Theorem \\
\rm\large Sungkon Chang
\end{center}

\begin{quote}
{\bf Abstract:}
By \zec's Theorem, every positive integer is uniquely written as a sum
of non-adjacent terms of the \fib\ sequence, and its converse states that
if a sequence in the positive integers has this property, it must be the \fib\ sequence. If we instead consider the problem of finding a monotone sequence with such a property, we call it the weak converse of \zec's theorem.
In this paper, we first introduce a generalization of \zec\ conditions, and
subsequently, \zec's theorems and their weak converses for the general \zec\ conditions. We also extend the generalization and results to the real numbers in the interval $(0,1)$, and to $p$-adic integers.

\end{quote}

\section{Introduction}\label{sec:intro}

\zec's Theorem \cite{zec}
states that each positive integer is expressed uniquely as a sum of distinct nonadjacent terms of the Fibonacci sequence $(1,2,3,5,\dots)$ where
we reset $(F_1,F_2)=(1,2)$.
Similar to the binary expansion, each positive integer can be
expressed as a sequence of $0$ and $1$ indicating whether
the Fibonacci term is involved or not.
For example, the natural number $100$
corresponds to the {\it\zec\ digits} $(1000010100)_Z$, meaning that
$ F_{10} +F_5+ F_3 = 89+ 8 + 3=100$.
\zec\ digits share the simplicity of representation with the binary expansion, but also they are quite curious in terms of the arithmetic operations, determining the $0$th digits, the partitions in \fib\ terms, and
the minimal summand property of \zec\ expansions; see
\cite{miller:minimality}, \cite{grabner}, \cite{kimberling},
\cite{fenwick}, and \cite{robbins}. One of the most striking features of \zec's Theorem is its converse and the questions it opens up.

\begin{theorem}[Daykin]\label{thm:daykin}
If a sequence $\set{Q_k}_{k=1}^\infty$ of positive integers
uniquely expresses each positive integer as a sum of its distinct non-adjacent terms, then it is the Fibonacci sequence.
\end{theorem}
This is called {\it the converse of \zec's Theorem}, and
we shall call the problem of finding monotone sequences rather than arbitrary sequences
{\it the weak converse of \zec's Theorem}.
In this paper, we introduce:
\begin{enumerate}
\item a general approach to
{\it \zec\ conditions}, generalizing the conditions introduced in \cite{mw};
\item \zec's theorem for a general \zec\ condition, which includes cases of linear recurrences with negative coefficients;
\item results on their weak converses, not only for
sequences in the positive integers, but also sequences in the real numbers and $p$-adic integers.
\end{enumerate}

A general \zec\ condition shall be properly introduced in Section \ref{sec:D-R}, and
in this section let us introduce another example to help the reader
be familiar with \zec\ conditions.
{\it The $N$th order Fibonacci sequence}
$\seq H$, whose name is coined in \cite{fraenkel}, is defined by
$H_n = H_{n-1}+\cdots + H_{n-N}$ for all $n>N$
and $H_n = 2^{n-1}$ for all $1\le n\le N$, and
\zec's Theorem for the $N$th order \fib\ sequence
states that each positive integer is expressed uniquely as a sum of distinct terms of the $N$th order Fibonacci sequence where no $N$ consecutive terms are used \cite{bruckman},
\cite{mw}.
We may call the restriction of not allowing $N$ consecutive terms {\it the $N$th order \zec\ condition}.
The weak converse for this \zec\ condition can be stated as follows:
{\it The $N$th order \fib\ sequence is the only increasing sequence that represents $\nat$ uniquely under the $N$th order \zec\ condition}, and
a proof is found in \cite{bruckman}.

Another interesting direction that the converse theorem opens up is
investigating
the unique existence of a sequence when a \zec\ condition and a set of numbers are given.
We say that a set $X$ of numbers is {\it represented by a sequence
$\seq Q$ uniquely under a \zec\ condition}
if each member of $X$ is uniquely expressed as a sum of terms of the sequence that satisfies the \zec\ condition, and each sum of terms of the sequence that satisfies the \zec\ condition is a member of $X$.
For example, we may ask whether
the set of positive odd integers can be represented by an increasing sequence under the second order \zec\ condition, and if so,
whether such a sequence uniquely exists, which is {\it the weak converse for the positive odd integers under the second order \zec\ condition}.
Let us introduce
another representative example of this direction of research.
Let $X$ be the open interval $(0,1)$ of real numbers, and
ask ourselves whether the interval
can be represented by a decreasing sequence of positive real numbers uniquely under the second order \zec\ condition, and if so,
does the weak converse for the interval under the second order \zec\ condition hold?
We shall provide answers to these two questions in Section \ref{sec:D-R} along with our main results which are presented in a more general setting.
The remainder of the paper is organized as follows.
In Section \ref{sec:D}, general definitions of \zec\ conditions are introduced along with results on its formulation in terms of blocks.
Introduced in Section \ref{sec:results} are main results
on \zec's Theorem and their weak converses
for  sets of numbers in $\nat$, the interval $(0,1)$ of real numbers, and
$p$-adic integers in 
$\zz_p$.
Examples are instrumental for properly understanding
the general concepts of \zec\ conditions, and they are briefly introduced in Section \ref{sec:D-R}.
However, it is necessary to discuss more examples that are interesting, in order to present the full extent of the definition, and
they are introduced in Sections \ref{sec:non-standard} and \ref{sec:examples}.
The main results introduced in Sections \ref{sec:D} and \ref{sec:results}
are proved in Section \ref{sec:proofs}.

\subsection*{Acknowledgement}
We thank Stephen J.\ Miller, Arturas Dubickas for answering our questions,
and thank the referees for carefully reading our manuscript.
We also thank Timothy Eller for inspiring questions and conversations that
initiated this project.

\section{Definitions and results}\label{sec:D-R}

\subsection{Definitions}\label{sec:D}
In this paper, $R$ will denote one of the following sets of numbers:
the natural numbers
$\nat$, the open interval $\OI:=(0,1)$ of real numbers, and the $p$-adic integers $\zz_p$.
A sequence  is usually denoted by $\set{a_n}_{n\ge 1}$.
In this paper, a sequence of numbers in $R$ is identified with a list of numbers
in the infinite product $\prod_{j=1}^\infty R$.
We usually denote them by capital letters such as $Q$, and their terms are denoted by $Q_k$ for $k=1,2 ,3\cdots$.
For example, if $Q_k=k$ for $k\ge 1$, then $Q=(1,2,3,\cdots)$.
Given a function $\ep: \nat \to \nat\cup\set{0}$ and a sequence $Q$,
we denote $\ep(k)$ by $\ep_k$, and define
$\sum \ep Q$ to be the formal sum $\sum_{k=1}^\infty \ep_k Q_k$.
In this context, $\ep$ is called a {\it\cf}.
We also use the list notation to present the values of $\ep$, i.e., $\ep=(\ep_1,\ep_2,\dots)$, and the bar notation $\bar a$ denotes the repeating entries,
e.g., $\ep=(1,2,3,\bar 0)$ meaning that $\ep_k=0$ for all $k>3$.
If there is an index $M$ such that $\ep_k=0$ for all $k >M$, $\ep$ is said to
{\it have a finite support}, and we say,
a \cf\ $\ep$ is {\it supported on a subset of indices $A$}
if $\ep_k=0$ for all $k\not \in A$.
Note here that given a \cf\ $\ep$, such an index subset $A$ is not uniquely determined, and it is a subset we assign to a \cf.
Let $\beta^i$ be the \cf\ such that $\beta^i_k=0$ for all $k\ne i$ and
$\beta_i^i=1$, and call it {\it the $i$th basis \cf}.

If a subset of indices $A$ consists of consecutive indices $\set{a,a+1,\dots,a+n}$, we call it an interval of indices.
Given an interval $J$ of indices and a \cf\ $\delta$, let both $\delta \reS J$ and $\resab{J}(\delta)$ denote $\sum_{k\in J} \delta_k \beta^k$, i.e., the restriction of $\delta$ on the indices
in $J$.
For example, $\resab{[a,b)}(\delta)=\sum_{k=a}^{b-1} \delta_k\beta^k$.
For convenience, let us denote
$\resab{[1,M]}{\delta}$ by $\delta \reS M$,
and the relationship
$\resab{J}{\delta} =\resab{J}{\mu}$
by $\delta \equiv \mu \reS J$ and also
$\delta \equiv \mu \reS [1,M]$
by $\delta \equiv \mu \reS M$.

Let us consider \lex\ orders on the set of \cf s.
Given two \cf s $\ep$ and $\ep'$, we define
{\it the descending \lex\ order} as follows.
If there is a smallest positive integer $k$ such that
$\ep_j = \ep'_j$ for all $j<k$ and 
$\ep_k< \ep'_k$, then we denote the property by $\ep \dord \ep'$.
For example, if $\ep=(1,2,10,5,\dots)$ and $\ep'=(1,3,1,10,\dots)$,
then $\ep \dord \ep'$ since $\ep_2 < \ep'_2$ and $\ep_1=\ep'_1$.
Let us point out that 
 the \lex\ order is defined on the set of \cf s, and it does not mean that 
 the values of a \cf\ in the set form a decreasing sequence.
For the representation of the real numbers in the open interval $\UI$, we shall use the descending \lex\ order on the set of \cf s.
Given two \cf s $\mu$ and $\mu'$ with finite support, we define
{\it the ascending \lex\ order} as follows.
If there is a largest positive integer $k$ such that
$\mu_j = \mu'_j$ for all $j>k$ and 
$\mu_k < \mu'_k$, then we denote the property by $\mu \aord \mu'$.
For example, if $\mu=(1,2,10,3,7)$ and $\mu'=(1,3,1,4,7)$,
then $\mu \dord \mu'$ since $\mu_4 < \mu'_4$ and $\mu_5=\mu'_5$.
As in the earlier case,
it does not mean that 
 the values of a \cf\ in the set form an increasing sequence.
For the representation of the positive integers we shall use the ascending \lex\ order on the set of \cf s with finite support.

Given a set of numbers $R$ listed above, we define
{\it
a collection $\cE$ of \cf s under a \lex\ order} to be
a set of \cf s ordered by the same \lex\ order that contains the zero \cf\ and all basis \cf s $\beta^i$.
We  call the set {\it an \aorcoll\ collection of   \cf s} if it is under
the ascending \lex\ order, and {\it a \dorcoll\ collection of   \cf s}
 if it is under
the descending \lex\ order.
A member of $\cE$ is called an $\cE$-\cf, and a \cf\ is said to {\it satisfy
the $\cE$-condition} if it is a member of $\cE$.
For example, if $\cE$ is the collection of \cf s $\mu$ with finite support such that $\mu_k$ is either $0$ or $1$ for all $k\ge 1$ and the list $\mu$ does not have two consecutive entries of $1$, then $\cE$ is the classical \zec\ condition used for writing positive integers as a sum of \fib\ terms.

Let $\cE$ be a collection of \cf s under the ascending or descending \lex\ order.
Let $\delta$ be a \cf\ in $\cE$, and we introduce the following terminology with respect to its \lex\ order.
The smallest \cf\ in $\cE$ that is greater than $\delta$,
if (uniquely) exists, is called {\it the \immsucc\ of $\delta$ in $\cE$}, and we denote it by $\td\delta$.
The largest \cf\ in $\cE$ that is less than $\delta$,
if (uniquely) exists, is called {\it the \immpred\ of $\delta$ in $\cE$}, and we denote it by $\hat\delta$.

Let us introduce an order notation that will be instrumental throughout the paper,
and it is intended to reflect the magnitude of a number
expressed in terms of \cf s.
Let $\ep$ be a non-zero   \cf\  of a collection under the descending \lex\ order, which will be used for the real numbers in $\OI$.
The smallest index $n$ such that $\ep_n\ne 0$ is called the order of $\ep$, denoted by $\ord(\ep)$.
If $\ep=0$, then we define $\ord(\ep)=\infty$.
For a non-zero  function $\mu$ of a collection under the ascending \lex\ order that has finite support,
the largest index $n$ such that $\mu_n \ne 0$ is called the order of $\mu$,
denoted by $\ord(\mu)$. If $\mu=0$, we define
$\ord(\mu)=0$.

Let us further introduce
the notion of {\it \zec\ collections of  \cf s under the ascending \lex\ order}, which will be used for positive integers.
By definition, a collection of \cf s contains all the basis \cf s, i.e.,  $\beta^{n-1}\in\cE$ for $n\ge 2$, and hence, the \immpred\ $\hbeta^n$,  if exists,
 has
a non-zero value at index $n-1$, i.e., $\ord(\hbeta^n)=n-1$.

\begin{deF}\label{def:zec-N}
Let $\cE$ be an \aorcoll\ collection of  \cf s with finite support.
The collection is called \zec\ for positive integers if it satisfies the following:
\begin{enumerate}
\item For each $\mu\in\cE$ there are at most finitely many \cf s that are less than $\mu$.

\item
Given $\mu\in \cE$, if
its \immsucc\ $\td\mu$ is not $ \beta^1+\mu $,
then there is an index $n\ge 2$
such that $ \mu \equiv \hbeta^n \reS [1,n)$ and $\td\mu = \beta^n
+ \resab{[n,\infty)}(\mu)$.

\end{enumerate}
\end{deF}
 Definition \ref{def:zec-N}, Part 2 says that  each \cf\ of a \zec\ collection
$\cE$ for positive integers
has a (unique) \immsucc\ in $\cE$, and 
Definition \ref{def:zec-N}, Part 1 implies that
 it has a (unique) \immpred\ as well. Let us use the following lemma to 
 explain this property.
\begin{lemma}\label{lem:seq}
Let $\cE$ be a \zec\ collection  for positive integers.
Let $\tau^0$ be the zero \cf, and let $\tau^{n+1}$ be the \immsucc\ of $\tau^n$ in   $\cE$   for each $n\ge 0$.  Then, $\tau^n$ is the \immpred\ of $\tau^{n+1}$ in $\cE$ for each $n\ge 0$, and $\set{\tau^n : n \ge 0} = \cE$.
\end{lemma}
\begin{proof}
Let $\tau^0$ be the zero \cf, and let $\tau^{n+1}$ be the \immsucc\ of $\tau^n$ for each $n\ge 0$.  Suppose that there are $\delta\in\cE$
and an integer $n\ge 0$ such that  $\tau^n \aord \delta \aord
 \tau^{n+1}$.
 Then, it contradicts that $\tau^{n+1}$ is the \immsucc\ of $\tau^n$.
 Hence, this proves that $\tau^n$ is the \immpred\ of $\tau^{n+1}$.
 
Notice that  
$S:=\set{\tau^n : n \ge 0}$ is a subset of $\cE$. 
Let $\mu$ be a \cf\ in $\cE$ that is greater than $\tau^0$, and let us show $\mu\in S$.
The subset $T:=\set{\al \in \cE : \al \aordeq \mu}$ is a finite set by Definition \ref{def:zec-N}, Part 1.
By the definition of the ascending \lex\ order, $\tau^0 \in T$, and
there is a largest element $\tau^m$    of the nonempty finite subset $S\cap T$.
It suffices to show that $T':=\set{\al \in   T : \tau^m \aord \al }$ is empty, i.e., $\tau^m$ is the largest element of $T$, which is $ \mu$.
Suppose that $T'$ contains an element  $\gamma$.
The collection $\cE$ is totally ordered under the ascending \lex\ order, i.e.,
either $\tau^{m+1} \aordeq \gamma$ or $\gamma \aord \tau^{m+1}$ is true.
If $ \tau^{m+1} \aordeq \gamma$, then $\tau^m < \tau^{m+1} \aordeq \gamma
\aordeq \mu$, which 
contradicts the choice of $\tau^m$.
If $\tau^m\aord \gamma \aord \tau^{m+1}$, then it contradicts that $ \tau^{m+1}$ is 
the \immsucc\ of $\tau^m$.
The implications of the above two cases contradict the existence of $\gamma\in\cE$ under the \lex\ order, and hence, we prove that $\tau^m =\mu\in S$.
\end{proof}

Before we introduce examples, let us extend the definition to $p$-adic integers.
A \cf\ $\mu$ is called the limit of a sequence of \cf s $\mu^k$ with finite support
for $k\ge 1$
if there is an increasing sequence of indices $M_k\ge \ord(\mu^k)$ for $k\ge 1$ such that
$\mu \equiv \mu^k \reS M_k $ for all indices $k\ge 1$.
A collection of \cf s $\cE$ is called {\it a \zec\ collection
for $p$-adic integers}
if it has a \zec\ sub-collection $\cE_0$ for positive integers
such that $ \cE$ is the set of \cf s that are the limits of sequences of \cf s in $\cE_0$, and
$\cE$ is also called the completion of $\cE_0$.

\begin{example}\label{exm:k-1}\rm
Let $\cE_0$ be the \aorcoll\ collection of  
\cf s $\mu$ with finite support
such that $\mu_k\le k$ and $\mu_{k+1}=k+1$ implies $\mu_k=0$ for all $k\ge 1$.
Then, $\cE_0$ is \zec\ for positive integers.
For example, $\hbeta^7=(0,2,0,4,0,6,\bar 0)$, and
the \immsucc\ of $\hbeta^7 + 3\beta^7 + 2\beta^8$
is $4\beta^7 + 2\beta^8$.
If $\cE$ is the completion of $\cE_0$, then
$ \sum_{k=1}^\infty (1+2k)\beta^{1+2k}$ and $ \sum_{k=1}^\infty (3k)\beta^{3k}$ are examples of \cf s in $\cE$.
\end{example}

\begin{example}\label{exm:zero-coeff}\rm
Let $B=\set{(a_1,a_2,a_3)\in\zz^3 : 0\le a_k \le 1\ \text{for all }k}
-\set{(1,1,0)}$,
and let $\cE$ be the \aorcoll\ collection of  \cf s $\ep$ with finite support
generated by concatenating some blocks in $B$.
Then, $\cE$ is a \zec\ collection for positive integers, 
whose \immpred s are given by $\hbeta^n=
\sum_{k=1}^{n-1} \beta^k=(1,1,\dots,1,1,\bar 0)$ for $n\not\equiv 0\mod 3$,
and $\hbeta^{3n}=\sum_{k=1}^{3(n-1)} \beta^k
+\beta^{3n-1}=(1,\dots,1,0,1,\bar 0)$ for $n\ge 1$.
\end{example}

Definition \ref{def:zec-N} accomplishes
a concise description of \zec\ conditions in terms of properties the collection must satisfy, and when it is easy to determine the \immsucc\ of each member as in Example \ref{exm:k-1}, it is useful for determining whether a collection is \zec\ or not.
However, it turns out that a \zec\ collection is completely determined by the subset $\set{\hbeta^n : n\ge 2}$, and it is not so simple
to see this fact, i.e., to find the \immsucc\ of $\delta$ when a \cf\ $\delta$ and $\hbeta^n$ for $n\ge 2$ are given.
Theorem \ref{thm:block-def-N} and Corollary \ref{cor:generate-zec} below will make this clear,
and the proof will be given in Section \ref{sec:proofs}.
Theorem \ref{thm:block-def-N} and Corollary \ref{cor:generate-zec} will also show that
Definition \ref{def:zec-N}
generalizes the definition introduced in \cite[Definition 1.1]{mw}.
 
\begin{deF}\label{def:E-block-N}
Let $\set{\delta^n : n\ge 2}$ be a set of \cf s  under the ascending \lex\ order such that $\ord(\delta^n) = n-1$ for all $n\ge 2$. 
A \cf\ $\zeta$   is called
{\it a proper $\delta$-block at index $n$}
if there is an index $1\le i\le n$ such that $0\le \zeta_i< \delta^{n+1}_i$,
$\zeta \equiv \delta^{n+1} \reS (i,n]$,
and $\zeta_s =0$ for all $s\not\in [i,n]$.
We call the interval of indices $[i,n]:=\set{i,\dots,n}$
{\it the support of
a proper $\delta$-block at index $n$}.
The \cf\ $\delta^n$ is  called {\it the maximal $\delta$-block at index $n-1$}, and we
call the interval of indices $[1,n)$ {\it the support of $\delta^n$}.

\end{deF}
Note that the zero \cf\ is declared to be a proper \dbl\ at index $n$ 
for any integer $n\ge 1$, for which the support is $[n,n]$, while  the support of a nonzero proper \dbl\ is uniquely determined.
The basis \cf\ $\beta^n$  is a simple example of   nonzero $\delta$-blocks at index $n \ge 1$, which may or may not be proper. Consider the collection $\cE_0$ defined in Example \ref{exm:k-1}.
Then, $\zeta=(0,0,0,0,5)$ is an example of  proper $\hbeta$-blocks, and its
 support is $[3,5]$ since $i=3$ is the index such that 
 $\zeta_i < \hbeta_i=3$ and $\zeta_k = \hbeta_k$ for $k=4,5$. 

Our main interest for positive integers is an \aorcoll\ collection $\cE$ of \cf s with finite support such that 
the \immpred s $\hbeta^n$ exist for each $n\ge 1$, and in general, 
 a $\hbeta$-block  $\zeta$ at index $n$ is not required to be a member of $\cE$.  
If $\cE$ is \zec, then by Theorem \ref{thm:block-def-N}, Part 2, a proper \Eb\ is a member of $\cE$.

\begin{theorem}\label{thm:block-def-N}
Let $\cE$ be an \aorcoll\ collection of   \cf s with finite support  such that 
the \immpred s $\hbeta^n$ exist for each $n\ge 2$.
\begin{enumerate}
\item
The collection $\cE$ is \zec\ if and only if
all of the following are satisfied:
\begin{enumerate}
\item
For each $\mu\in\cE$, there are a positive integer $M$, a unique \Eb\ $\zeta^1$ with support $[i_1,n_1]$, and unique proper \Eb s $\zeta^m$ with support $[i_m,n_m]$ for $2\le m\le M$ (if $M\ge 2$)
such that $i_1=1$, $n_m+1=i_{m+1}$ for all $1\le m\le M-1$, and
$\mu = \sum_{m=1}^M \zeta^m$.
We call the expression the \Eb\ decomposition.

\item Given $\mu\in\cE$, if $\mu = \sum_{m=1}^M \zeta^m$ is the
\Eb\ decomposition, then $\td\mu = \beta^1 + \mu$ if $\zeta^1$ is not maximal, and $\td\mu = \beta^n +\sum_{m=2}^M \zeta^m$ if $\zeta^1 = \hbeta^n$ for some $n\ge 2$.

\end{enumerate}
\item Let $\cE$ be \zec,
and let
$\zeta^m$ for $m=1,\dots,M$ be a sequence of \Eb s with
disjoint supports $[i_m,n_m]$ such that $\zeta^m$ are proper for
$m\ge 2$ (if $M\ge 2$).
Then, the \cf\ $ \sum_{m=1}^M \zeta^m$ is a member
of $\cE$.
\end{enumerate}

\end{theorem}
If $\cE$ is \zec, then for convenience we may write the \Eb\ decomposition as
$\ep = \sum_{m=1}^\infty \zeta^m$ where
all sufficiently large \Eb s are zero \cf s.
Let us introduce some examples.
The collection $\cE$ for the classical \zec\ condition is
a \zec\ collection, e.g.,
$\mu=(0,1,0,1,0,0,1,\bar 0)$
is decomposed into non-zero \Eb s $(0,1,0,1,\bar 0) + \beta^7$ where
$\zeta^1=(0,1,0,1,\bar 0)$ is the maximal \Eb\ at index $4$,
and $\td\mu=\beta^5 + \beta^7$.
The \cf s $\ep:=\sum_{k=1}^\infty\beta^{n+2k}$ for $n\ge 1$ are members of the completion $\bar\cE$.
For the completion of $\cE$,
an infinite sum of proper \Eb s with disjoint supports is a member of $\cE$, but some members of $\cE$ such as $\ep$ defined above may not be written as an infinite sum of proper \Eb s with disjoint supports.

Another important application of Theorem \ref{thm:block-def-N} is constructing a \zec\ collection 
with the \immpred s $\hbeta^n$. 
\begin{deF}\label{def:generate-zec}
Let $\set{\delta^n : n\ge 2}$ be a set of \cf s  under the ascending \lex\ order such that $\ord(\delta^n) = n-1$ for all $n\ge 2$. 
 {\it The  \aorcoll\ collection $\cE$ of   \cf s determined  by $\set{\delta^n : n\ge 2}$} is defined to be the \aorcoll\ collection $\cE$ of   \cf s consisting of
 $\delta$-block decompositions $ \sum_{m=1}^M \zeta^m$ where 
 $M$ is a positive integer.  
 \end{deF}
\begin{cor}\label{cor:generate-zec}
If $\cE$ is the \aorcoll\ collection  of   \cf s determined by \cf s $\delta^n$ of order $n-1$ for $n\ge 2$, then $\cE$ is \zec, and $\hbeta^n = \delta^n$ for all $n\ge 2$.
Moreover, $\cE$ is the only \zec\ collection for positive integers such that 
$\hbeta^n=\delta^n$ for each $n\ge 2$.
\end{cor}

By Corollary \ref{cor:generate-zec}, given a
set of \cf s $\delta^n$ of order $n-1$ for $n\ge 2$,
we may define  
 {\it the \zec\ collection for positive integers with \immpred s
$\hbeta^n=\delta^n$ for $n\ge 2$} to be the collection defined in Definition
\ref{def:generate-zec}.

Let us introduce \zec\ collections  for the unit interval $\OI$.

\begin{deF}\label{def:immsucc-R}

Let $\cE$ be a \dorcoll\ collection of   \cf s,
and given an index $M\ge 1$, let
$\cE^M$ denote the collection consisting
of $\ep \reS M$ for $\ep\in\cE$.
The collection $\cE$ is called \zec\ for the open interval $\OI$ if the following
are satisfied:
\begin{enumerate}

\item For each $\mu\in\cE^M$
there are at most finitely many \cf s in $\cE^M$ that are less than $\mu$.
\item Given an index $n\ge 1$, there is a unique \cf\ $\bbeta^n$
of order $ n$ with infinite support, not necessarily a member of $\cE$,
such that $\bbeta^n \reS M$ is the \immpred\ of $\beta^{n-1} \reS M$
in $\cE^M$
for all $M\ge n$ if $n\ge 2$, and it is maximal in $\cE^M$ for all $M\ge n$ if $n=1$. The \cf s $\bbeta^n$ are called the maximal \cf\ of order $ n$ for $\cE$.

\item
Given $\mu\in \cE^M$ that is less than $\bbeta^1 \reS M$,
if
its \immsucc\ $\td\mu$ in $\cE^M$ is not $\mu + \beta^M$,
then there is an index $1\le n< M$
such that $\bbeta^{n+1}\equiv \mu \reS (n,M]$ and $\td\mu =\resab{[1,n]}(\mu) + \beta^n $.

\item
Let $\ep$ be a   \cf\ in $\cE$.
Then, $\ep \in \cE$ if and only if there are infinitely many indices $M\ge 1$ such that
$\ep^M:=\ep \reS M$ is a member of $\cE^M$ and
the \immsucc\ of $\ep^M$ in $\cE^M$ is given by $\ep + \beta^M \reS M$ .

\end{enumerate}
\end{deF}
\begin{example}\label{exm:OI}\rm
Let $\cE $ be the \dorcoll\ collection of  
\cf s $\mu$ 
such that $\mu_k\le k$ and $\mu_{k}=k$ implies $\mu_{k+1}=0$ for all $k\ge 1$.
Then, $\cE$ is \zec\ for $\OI$.
The \cf\ $\bbeta^3=(0,0,3,0,5,0,7,0,\cdots)$
is an example of maximal \cf s for $\cE$, and it is not a member of $\cE$.
The \immsucc\ of $\resab{[1,6]}(\beta^2 + \bbeta^3)=(0,1,3,0,5,0)$ in $\cE^6$ 
is $2\beta^2$. 
\end{example}
Notice that since $\beta^n \in\cE^M$ for all $M\ge n$,
the property $\resab{[1,M]}(\beta^n) \dord \resab{[1,M]}(\bbeta^n)$ implies that $\bbeta^n_n\ge 1$, i.e., $\ord(\bbeta^n)=n$.
Also notice that $\bbeta^n$
does not satisfy the existence of infinitely many indices $M$ described
in Definition \ref{def:immsucc-R}, Part 4, and hence,
it is not a member of the \zec\ collection $\cE$.
This condition is motivated from the situation where we have
two representations in the binary expansions,
$1/2 = 1/2^2 + 1/2^3 + 1/2^4+\cdots$.
However, the fact that a   \cf\ $\mu$ terminates with a maximal \cf\ of order $n$
does not imply that $\mu \not\in \cE$.
For example, $\mu:=\beta^1 + \bbeta^3$ may or may not be members of $\cE$, and we shall explain this properly  after Corollary \ref{cor:generate-zec-real} below.
As in the case of Definition \ref{def:zec-N},
a \zec\ collection for $\OI$ is
completely determined by the maximal \cf s of order $n$ for $n\ge 1$, and
it is proved by
Theorem \ref{thm:block-R} and Corollary \ref{cor:generate-zec-real} below.

\begin{deF}
Let $\set{\delta^n : n\ge 1}$  be a set of \cf s with infinite support  under the descending \lex\ order such that $\ord(\delta^n)=n$ for each $n\ge 1$.
A   \cf\ $\zeta$  is called
{\it a proper $\delta$-block at index $n$}
if there is an index $ i \ge n$ such that $0\le \zeta_i< \delta^n_i$,
$\zeta\equiv \delta^n \reS [n,i)$, and $\zeta_s=0$
for all $s\not\in [n,i]$,
and the interval of indices $[n,i]$ is called
{\it the support of
a proper $\delta$-block at index $n$}.

\end{deF}
The zero \cf\ is declared to be a proper $\delta$-block at any index $n$ with the support $[n,n]$, and the basis \cf\ $\beta^n$ is a proper $\delta$-block at index $n$ for any integer $n\ge 1$ since $\delta^n$ has infinite support.
Let $\cE$ be the \dorcoll\ collection defined in Example \ref{exm:OI}, and 
let $\mu =(0,0,3,0,5,0,0,\bar 0)$.
Then, $\mu\in\cE$, and it is a proper \Ebb\ at index $3$ with support $[3,7]$.

Our main interest is a descendingly-ordered collection $\cE$ of 
\cf s such that the \cf\ $\bbeta^n$ defined in Definition \ref{def:immsucc-R}, Part 2 exists for each $n\ge 1$.
In general, a proper \Ebb\ is not required to be a member of $\cE$.
By Theorem \ref{thm:block-R}, Part 2 below, if $\cE$ is \zec,
then the proper \Ebb s are members of $\cE$.

\begin{theorem}\label{thm:block-R}
Let $\cE$ be a \dorcoll\ collection   \cf s for the open interval $\OI$.

\begin{enumerate}

\item
The collection $\cE$ is \zec\ if and only if
all of the following are satisfied:

\begin{enumerate}
\item Given an index $n\ge 1$, there is a maximal \cf\ $\bbeta^n$ of order $ n$ as defined in Definition \ref{def:immsucc-R}.

\item
For each $\ep\in\cE$, there are unique proper \Ebb s $\zeta^m$
at index $n_m$ with support $[n_m,i_m]$ for all $m\ge 1$
such that $n_1=1$ and $i_m+1=n_{m+1}$, and
$\ep = \sum_{m=1}^\infty \zeta^m$.
We call the expression the \Ebb\ decomposition of $\ep$.

\item Given a  \Ebb\ decomposition $\ep = \sum_{m=1}^\infty \zeta^m$ of $\ep\in\cE$ and an index $M\ge 1$,
the \immsucc\ $\td\ep$ of $\ep \reS M$ in $\cE^M$ is given as follows, if $\ep \not
\equiv \bbeta^1 \reS [1,M]$.
Let $K$ be the index such that the support $[n_K,i_K]$ of $\zeta^K$ contains $M$. If $n_K\le M < i_K$, then $\td\ep= \sum_{m=1}^{K-2} \zeta^m + \zeta^{K-1} + \beta^{n_K-1}$.
If $M=i_K$, then
$\td\ep= \sum_{m=1}^{K-1} \zeta^m + \zeta^{K} + \beta^{i_K}$.

\end{enumerate}

\item Suppose that $\cE$ is \zec, and let $\zeta^m$ be proper \Ebb s
at index $n_m$ with support $[n_m,i_m]$
such that $n_1=1$ and $i_m+1=n_{m+1}$ for all $m\ge 1$.
Then,
$\sum_{m=1}^\infty \zeta^m$ is a member of $\cE$.
\end{enumerate}

\end{theorem}

As in the case of  \zec\ collections for positive integers, a \zec\ collection
$\cE$ for $\OI$ is completely determined by the set of maximal \cf s $\bbeta^n$. 
\begin{deF}\label{def:generate-zec-real}
Let $\set{\delta^n : n\ge 1}$  be a set of \cf s with infinite support  under the descending \lex\ order such that $\ord(\delta^n)=n$ for each $n\ge 1$.
 {\it The  \dorcoll\ collection $\cE$ of   \cf s determined by $\set{\delta^n : n\ge 1}$} is defined to be the \dorcoll\ collection $\cE$ of   \cf s consisting of
$\delta$-block decompositions $ \sum_{m=1}^\infty \zeta^m$.
 \end{deF}
\begin{cor}\label{cor:generate-zec-real}  
If $\cE$ is the \dorcoll\ collection $\cE$ of   \cf s determined by \cf s $\delta^n$ of order $n$ with infinite support for $n\ge 1$, then $\cE$ is \zec, and $\bbeta^n =\delta^n$ for all $n\ge 1$.
Moreover, $\cE$ is the only \zec\ collection for $\OI$ such that 
$\bbeta^n=\delta^n$ for each $n\ge 1$.
\end{cor}
By Corollary \ref{cor:generate-zec-real},
given a set of \cf s $\delta^n$ of order $n$ with infinite support    for  $n\ge 1$,
we may define 
{\it the \zec\ collection for $\OI$ determined by maximal \cf s $\bbeta^n=\delta^n$ for $n\ge 1$} 
to be the collection defined in Definition \ref{def:generate-zec-real}.
For example, let $\cE$ be the  \zec\ collection for $\OI$ determined by 
maximal \cf s $\bbeta^n = \sum_{k=0}^\infty \beta^{n+2k}
=(\bar 0,1,0,1,0,1,0,\dots)$ for $n\ge 1$. Then, the collection is similar to the one for the classical \zec\ condition, but
it allows infinitely many entries of $1$.

Let us revisit the earlier example $\mu:=\beta^1 + \bbeta^3$.
Suppose that $\bbeta^3=\sum_{k=0}^\infty \beta^{3+2k}$
and $\bbeta^n = \sum_{k=n}^\infty \beta^k$ for all $n\ge 4$
are the maximal \cf s of a \zec\ collection for $\OI$.
If
$\bbeta^1 = \beta^1+\sum_{k=3}^\infty \beta^k$,
then
$\mu =(\beta^1+\beta^3 +0) + (\beta^5+0) +(\beta^7+0)+\cdots$ is an \Ebb\ decomposition, and hence, $\mu\in\cE$ by Theorem
\ref{thm:block-R}.
If $\bbeta^n=\sum_{k=n}^\infty \beta^k$ for all $n\ge 1$, then $\mu=\beta^1 + \sum_{k=3}^\infty \beta^k$ does not have an \Ebb\ decomposition since
the first proper \Ebb\ in $\mu$ is $(\beta^1+0)$,
but there is no proper \Ebb\ at index $3$ in $\mu$.
Hence, $\mu$ is not a member of $\cE$.

Let
\begin{equation}\label{eq:L-recursion}
Q_n = e_1 Q_{n-1} + \cdots + e_N Q_{n-N}
\end{equation}
be a linear recurrence for a sequence in $R$
where $N$ is a fixed positive integer and $e_k$ are integers independent of $n$.
Let us review the standard \zec\ conditions on the \cf s associated with this recursion for sequences $Q$ in $\nat$ in terms of \immpred s.
The conditions in full generality are first introduced in \cite{mw}, and introduced below would be a slight generalization toward adapting infinite expansions of numbers in $\OI$ and $\zz_p$.
Let $L=(e_1,\dots, e_N)$ be a finite list of non-negative integers where $e_1e_N>0$,
and we shall call it {\it a \zml}.
Given an index $n\ge 2$, let
$\delta^n =\sum_{k=1}^{n-1}\hat e_{\rem(k)} \beta^{n-k}
=(\dots, e_2,e_1,\hat e_N,e_{N-1},\dots, e_2,e_1,\bar 0)$ be
\cf s where $\rem(k)$ denotes the least positive residue of $k$ mod $N$
and $\hat e_k = e_k$ for all $1\le k<N$ and $ \hat e_N:=e_N -1 $.
We may consider the \zec\ collection for positive integers with 
$\hbeta^n=\delta^n$ for $ n\ge 2$,  and we denote it by $\cL$.
The completion of $\cL$ for the $p$-adic integers is denoted by
$\overline{\cL}$.
For the \zec\ collection   for the open interval $\OI$, we consider $\bbeta^n: =\sum_{k=n}^{\infty}\hat e_{\rem(k-n+1)} \beta^{k}
=(\bar 0,e_1,e_2,\dots,\hat e_N, e_1,e_2,\dots)$,
and we also denote the collection by $\cL$.

Given a set of numbers $R$ and a \lex ly ordered collection $\cE$ of
 \cf s,
a sequence $Q $ in $R$ is said to
{\it have the \EEunique}
if $\sum \delta Q$ have distinct values for the \cf s $\delta \in \cE$.
Given a sequence $Q$ with \EEunique,
let us denote by $X_Q^{\cE}$ the subset
consisting of values of $\sum \delta Q$ for non-zero \cf s $\delta\in \cE$, and we call it
an $\cE$-subset of $R$.
When the collection $\cE$ is understood in the context,
we simply denote the subset by $X_Q$.
Recall that if $L=(1,1)$ is a \zml, the \aclz\ condition on \cf s for positive integers coincides with the classical \zec\ condition
on the Fibonacci sequence.
For example, consider $Q_n=2^{n-1}$ for $n\ge 1$ under the \aclz\ condition.
Then, $165= Q_8 + Q_6 + Q_3 + Q_1$ is a member of $X_Q$
while $166$ and $167$ are not.
However, if $\widetilde L=(1,2)$, then the binary expansion of a positive integer is
$\widetilde\cL$-\zec, and $X_Q^{\widetilde\cL}=\nat$
while $X_Q^{\cL}$ is a proper subset of $\nat$.
Let us introduce another interesting example.
Let $L=(1,1)$, and let $Y$ be the subset consisting of $ \sum \mu F $
for \clz\ \cf s $\mu$ with $\mu_1=0$ where $F=(1,2,3,\dots)$ is the Fibonacci sequence, i.e.,
the positive integers whose classical \zec\ decompositions do not involve $F_1$.
Then, obviously, $Y$ is represented uniquely under the \Lz\ condition by the sequence $Q$ given by $Q_k:=F_{k+1}$ for $k\ge 1$.
More interestingly, by Theorem \ref{thm:integers} below,
it turns out that it is the only increasing sequence with that property.

Given a \lex ly ordered collection $\cE$ and a subset $Y$ of $R$, if there is a sequence $Q$ in $R$ with \EEunique\ such that $X_Q=Y$,
then $Q$ is called a {\it fundamental sequence for the $\cE$-subset $Y$}.
In addition, if the fundamental sequence $Q$ is an increasing sequence,
then $Y$ is also called an {\it increasing $\cE$-subset of $\nat$}.
If the fundamental sequence $Q$ is a decreasing sequence in $R$ that is either $\OI$ or $\zz_p$, the subset $Y$ is called a {\it decreasing $\cE$-subset of $R$}.

\iffalse
Also we say that {\it an $\cE$-subset $Y$ is generated by a fundamental sequence $Q$}, and
{\it a fundamental sequence $Q$ generates an $\cE$-subset $Y$}.
\fi

\subsection{Results} \label{sec:results}
Let us begin with the results on positive integers.
In this paper, for simplicity an increasing sequence
means a strictly increasing sequence in the usual sense.
\begin{theorem}
\label{thm:integers}
\phantom{nothing}
\begin{enumerate}
\item (\zec's Theorem for positive integers) Let $Q$ be an increasing sequence in $\nat$ with $Q_1=1$.
\begin{enumerate}
\item
Then, there are \cf s $\mu^n$ of order $n-1$ for $n\ge 2$ such that
\begin{equation}\label{eq:full-recursion}
Q_n = Q_1 + \sum\mu^n Q.
\end{equation}
\item Suppose that a sequence of \cf s
$\mu^n$ of order $n-1$ for $n\ge 2$ satisfies the recursion
(\ref{eq:full-recursion}) for all $n\ge 2$, and let $\cE$ be the \zec\ collection with $\hbeta^n:=\mu^n$ as the \immpred\ of $\beta^n$ %
for $n\ge 2$. Then, $Q$ is a fundamental sequence for the $\cE$-set of numbers $\nat$.
\end{enumerate}
\item (The weak converse)
Let $\cE$ be an arbitrary  \aorcoll\ collection of   \cf s with finite support.
\begin{enumerate}
\item If $\cE$ is \zec, then $\nat$ is an $\cE$-set of numbers.
\item
Each increasing $\cE$-subset $Y$ of $\nat$ has a unique
increasing fundamental sequence $Q$.
In addition, if $\cE$ is \zec, and $Y=\nat$, then $Q$ is given by
$Q_n=\sum\hbeta^{n} Q +1$ for all $n\ge 2$ and $Q_1=1$ where
$\hbeta^n$ for $n\ge 2$ are the \immpred s of $\beta^n$.
\end{enumerate}
\end{enumerate}
\end{theorem}

\begin{example}\label{exm:L-recursion}\rm
If $\cL$ is a \zec\ collection for $\nat$ determined by \zml\ $L=(e_1,e_2,\dots,e_N)$,
the recursion (\ref{eq:full-recursion})
is reduced to $Q_n = e_1 Q_{n-1} + \cdots e_N Q_{n-N}$
for all $n>N$, and
\Lz's Theorem is proved in \cite{mw}.
In Theorem \ref{thm:integers}, Part 2 (b), it is asserted that
it is the only increasing fundamental sequence.
\end{example}

Recall the \zec\ collection $\cE$ in Example \ref{exm:k-1}.
Then,
the increasing fundamental sequence $Q$ is given by
$Q_{n+2} = (n+1) Q_{n+1} + Q_n$ for all $n\ge 1$ and
$Q=(1,2,5,17,\dots)$.
Below we list the first few small \cf s in $\cE_0$:
\begin{gather*}
(1) \aord (0,1) \aord (1,1) \aord (0,2) \aord
(0,0,1) \aord (1,0,1) \aord (0,1,1) \aord (1,1,1)\\
\aord (0,2,1)\aord (0,0,2)\aord\cdots \aord (0,2,2)
\aord (0,0,3) \aord (1,0,3)\aord (0,0,0,1)
\end{gather*}
where the \immpred s $\hbeta^n$ for $n=2,\dots,5$ are given by
$(1)$, $(0,2)$, $(1,0,3)$, and $(0,2,0,4)$.
In terms of the values of $\sum\ep Q$,
each value is obtained by adding $Q_1$ to $\sum \hat\ep Q$
where $\hat\ep$ is the immediate predecessor of $\ep$.
Presenting values of $\sum\ep Q$
with respect to the \lex\ order of $\ep$ as above trivially proves that $Q$ is a fundamental sequence.

\begin{example}\label{exm:neg-coeff}\rm
Let $\cE$ be the \zec\ collection for positive integers with
$\hbeta^n=\sum_{k=1}^{n-2} \beta^k + 2\beta^{n-1}
=(1,1,\dots,1,1,2)$ for $n\ge 2$. Then, the increasing fundamental sequence is given by $Q_n = 3Q_{n-1}-Q_{n-2}$
for $n\ge 3$ and $Q=(1,3,8, 21, 55,\dots)$.
Below we list the first few small \cf s in $\cE$ determined by the \immpred s:
\begin{gather*}
(1) \aord (2) \aord (0,1) \aord (1,1) \aord (2,1) \aord
(0,2) \aord (1,2) \aord (0,0,1) \aord (1,0,1)\\
\aord (2,0,1)\aord (0,1,1)\aord\cdots \aord (2,0,2)
\aord (0,1,2) \aord (1,1,2)\aord (0,0,0,1)
\end{gather*}
The fundamental sequence in fact
satisfies the equation (\ref{eq:full-recursion}), and
the common \zquote{tail part} $\sum_{k=1}^{n-2} Q_k$
of the equation allows us to derive the short recursion above.
\end{example}

By reversing the process of finding a linear recurrence from \immpred s
with simple periodic tails as in Example \ref{exm:neg-coeff}, we obtain the following general result,
and the proof follows immediately from Proposition \ref{prop:neg-coeff} in Section \ref{sec:neg-coeff}.
\begin{theorem}\label{thm:neg-coeff}
Let $N$ be a non-negative integer, and
let $Q_n = \sum_{k=1}^{N+1} c_k Q_{n-k}$ for all $n\ge N+2$
be a linear recurrence
where $c_k$ are constants in $\zz$ such that $ (c_{k} + \cdots + c_{1}) \ge 1$ for all $1\le k\le N+1 $ and $c_1\ge 2$.
Then, there is a \zec\ collection $\cE$ and fixed non-negative coefficients
$e_1,\dots,e_N$ and $b$ for which $e_1\ge 2$,
$\hbeta^n = \sum_{k=1}^{n-N-1} b\beta^k
	+ \sum_{k=1}^N e_k \beta^{n-k}
	$ for all $n\ge N+2$, and
	$\hbeta^n = \sum_{k=1}^{n-1} e_k \beta^{n-k}$
	for all $2\le n\le N+1$, and
	there are initial values
	$(Q_1,\dots,Q_{N+1})$ for which the recurrence defines a fundamental sequence for $\nat$ under the \Ez\ condition.
\end{theorem}

For example,
let $Q$ be a sequence given by $Q_n
=8Q_{n-1} -2Q_{n-2}-3 Q_{n-3}$ for $n\ge 4$ and $(Q_1,Q_2,Q_3)=(1,8,62)$.
Then, the \zec\ collection described in Proposition \ref{prop:neg-coeff} and Theorem \ref{thm:neg-coeff} has $\hbeta^n = \sum_{k=1}^{n-3} 2 \beta^k
+ 5 \beta^{n-2} + 7 \beta^{n-1}$ for $n\ge 3$,
$\hbeta^3=5\beta^1 +7\beta^2$, and $\hbeta^2=7\beta^1$.
If coefficients $c_k$ do not satisfy the conditions in
Theorem \ref{thm:neg-coeff},
but there are increasing initial values $(Q_1,\cdots,Q_{N+1})$ for which
the recurrence defines an increasing sequence,
there is still a \zec\ condition under which $Q$ is an increasing fundamental sequence for $\nat$, as asserted in Theorem \ref{thm:integers}.
It will be further explained in Algorithm \ref{alg:E} below, but
we do not expect that its \immpred s have periodic tails.
Theorem \ref{thm:integers}, Part 1 (a) can be proved by
a trivial sequence of \cf s $\mu^n$ for $n\ge 2$ given by
\begin{equation}
\mu^n:=(Q_n- Q_{n-1} - 1)\beta^1 + \beta^{n-1} , \label{eq:trivial-hbeta}
\end{equation}
but
we may use a greedy algorithm as described in Algorithm \ref{alg:E} to
find $\mu^n $ where the values of
$\mu^n_k$ are relatively smaller.

\begin{alg}\label{alg:E}\rm
Let $Q$ be an increasing sequence in $\nat$ such that $Q_1=1$.
Given $Q_n$ for $n\ge 2$, let $e_1$ be the largest integer
such that $(Q_n-1) - e_1Q_{n-1}\ge 0$, and for $2\le k \le n-1$, recursively define
$e_k$ to be the largest integer such that
$(Q_n -1) - (e_1 Q_{n-1} +\cdots + e_{k} Q_{n-k} )\ge 0$.
Then,
since $Q_1=1$, the algorithm terminates with an equality
$Q_n=1+\sum_{k=1}^{n-1} e_k Q_{n-k}$.
By Corollary \ref{cor:generate-zec},
we have the \zec\ collection $\cE$ for positive integers with $\hbeta_n:=(e_{n-1},e_{n-2},\dots,e_1,\bar 0)$ for $n\ge 2$  where $(e_1,\dots,e_{n-1})$ depend on $n$,
and by Theorem \ref{sec:results},  the sequence $Q$ is the only increasing fundamental sequence for the $\cE$-set of numbers $\nat$.
\end{alg}

For example, let $Q$ be a sequence given by $Q_k=k!$ for all $k\ge 1$.
If we apply Algorithm \ref{alg:E} using the identity $1!\cdot 1 + 2!\cdot 2 + \cdots + n!\cdot n=(n+1)! - 1$, we obtain 
the \zec\ collection $\cE$ for positive integers with $\hbeta^n = \sum_{k=1}^{n-1} k \beta^k=(1,2,\dots,n-1,\bar 0)$ for all $n\ge 2$.
The trivial example of $\mu^n$ mentioned in (\ref{eq:trivial-hbeta}) yields
\immpred s
$\hbeta^n:=
(Q_n - Q_{n-1}-1)\beta^1 + \beta^{n-1}=
( (n-1)!\cdot(n-1)-1)\beta^1 + \beta^{n-1}$ for $n\ge 2$.

Let us also introduce the algorithm of finding the $\cE$-expansion of a positive integer
when $\cE$ is \zec.

\begin{alg}\label{alg:E-expansion}\rm
Let $\cE$ be a \zec\ collection for positive integers, and let $Q$ be the increasing fundamental sequence $Q$ for the $\cE$-set of numbers $\nat$.
If $x\in \nat$, the unique non-zero \Eb\ decomposition $x=\sum_{m=1}^M \sum \zeta^m Q$ is given by a greedy algorithm in terms of \Eb s. That is,
for each $k=1,2,\dots, M-1$, the proper \Eb\ $\zeta^{M-k+1}$ is recursively defined to be the largest one for which
$x-\sum_{m=1}^{k-1} \sum \zeta^{M-m+1} Q\ge \sum\zeta^{M-k+1} Q$, and $\zeta^1$ is the largest \Eb, proper or maximal, for which
$x-\sum_{m=1}^{M-1} \sum \zeta^{M-m+1} Q= \sum\zeta^1 Q$.
It is a straightforward induction exercise on $x$ to show that
the algorithm terminates with zero remainder.
\end{alg}

To prove Theorem \ref{thm:integers}, Part 1 (b),
we may use Algorithm \ref{alg:E-expansion}. In fact, we use the algorithm to establish the existence and uniqueness
of the $\cE$-expansions of real numbers in the interval $\OI$, and
the proof directly translates to the case of $\nat$.
In Section \ref{sec:proofs}
we also introduce a different approach to proving the case of $\nat$,
for which we do not use a greedy algorithm at all.

At first,
we were motivated to prove the weak converse for the $\cL$-subsets of $\nat$, and realized that the weak converse for subsets of $\nat$ holds even for non-\zec\ collections of \cf s, as stated in Theorem \ref{thm:integers}, Part 2 (b).
However, we learned later that Theorem \ref{thm:integers}, Part 2 (b) for non-\zec\ collections was
noticed and proved for $\nat$ in \cite{keller}, and
apparently it had been unnoticed in the later literature such as
\cite{bruckman} where
the weak converse for the $N$th order \zec\ set of numbers $\nat$
is proved using a different method.
Our proof is nearly identical to the one in \cite{keller}, but
we include our version in Section \ref{sec:proofs} as it is more generally for subsets of $\nat$.

Theorem \ref{thm:integers}, Part 2 concerns proper $\cE$-subsets of $\nat$ as well.
Let $L=(1,1,1)$ be a \zml.
Then, the \zec\ collection $\cL$ defines the $3$rd order \zec\ condition.
Let $Y:=7\nat$, i.e., the subset of positive multiples of $7$.
Then, $Y=X_Q$ where $Q$ is given by $Q_k=7 H_k$ where $H$ is the third order Fibonacci sequence $H=(1,2,4,7,13,\cdots)$.
Thus, $7\nat$ is an $\cL$-subset of $\nat$, and
by Theorem \ref{thm:integers},
the sequence $Q$ is the only increasing fundamental sequence for $7\nat$.
It turns out that other congruence classes mod 7 are not $\cE$-subsets of $\nat$ for any nontrivial \zec\ collections $\cE$, and neither is the subset of positive odd integers, which was
mentioned in Section \ref{sec:intro}.
We shall further discuss this example and more in Section \ref{sec:examples}.

Recall from Example \ref{exm:L-recursion} the fundamental sequence $Q$ for the $\cL$-set of numbers $\nat$.
If we instead choose initial values such that
$Q_n>\sum_{k=1}^{n-1} e_k Q_{n-k}$ for all $2\le n \le N$,
then it is a straightforward exercise to show that
$\mu \aord \delta$ implies $\sum \mu Q < \sum \delta Q$.
This implies that
$Q$ is an increasing sequence, and has the \Lunique.
We may also use the following general criteria to generate
more $\cL$-subsets of $\nat$.
If the initial values for the linear recurrence in Example \ref{exm:L-recursion} are increasing,
then the fundamental sequence is increasing, and by Theorem \ref{thm:integers}, Part 2, these increasing $\cL$-subsets satisfy the weak converse of \Lz's theorem.
\begin{theorem}\label{thm:unique}
Let $N$ be a positive integer, and let $Q$ be a sequence in $\nat$ given by the linear recurrence (\ref{eq:L-recursion}) where
$e_k\ge 1$ for all $k=1,\dots,N$.
Then, the sequence $Q$
has the \Lunique\ if and only if
the values of
$\sum \sig Q$ are distinct for all $\sig\in \cL$ of order $\le 4N$.
\end{theorem}

Let us demonstrate the theorem with $L=(2,3)$ and
$Q=(5,3,21,\dots)$.
As mentioned earlier,
the property
$Q_2 > 2 Q_1$ guarantees that
the sequence has the \Lunique, but
the inequality fails for $Q$.
However, computer calculations show that
there are 6,560 \Lz\ \cf s $\mu$ of
order $\le 8$, and they generate distinct values of $\sum \mu Q$.
Hence, by Theorem \ref{thm:unique},
it has the \Lunique.
See Section \ref{sec:proof-unique-property} for example of a sequence that does not have the \Lunique.
If $\cL$ is the $N$th order \zec\ condition,
the search can be shortened to the \cf s of order $\le 2N$; see \cite{chang}.

\begin{example}\label{exm:converse-fail}\rm
Let us demonstrate the the full converse of \zec's theorem fails for some
\zec\ conditions, i.e.,
there are two fundamental sequences for $\nat$ under the \zec\ condition.
Let $\cE$ be the \zec\ collection of   \cf s defined in
Example \ref{exm:zero-coeff},
and let $Q$ and $Z$ be sequences in $\nat$
defined by the recurrence $Q_n = 7Q_{n-3}$ and $Z_n = 7 Z_{n-3}$
such that $Q=(1,2,3,\dots)$ and $Z=(2,1,3,\dots)$.
Then, both sequences are fundamental sequences for the $\cE$-set of numbers
$\nat$, and $Z_{1+3n} > Z_{2+3n}$ for all $n\ge 0$.
\end{example}

Let us introduce our results for the unit interval $\OI$ of real numbers.
\begin{theorem} \label{thm:real} \phantom{nothing}
\begin{enumerate}
\item
(\zec's Theorem for the unit interval)
Let $Q$ be a decreasing sequence in $\OI$, and let $Q_0=1$.
\begin{enumerate}
\item If $Q_n\to0$
as $n\to\infty$, then there are   \cf s $\mu^n$
of order $n$
with infinite support for $n\ge 1$ such that
\begin{equation}\label{eq:full-recursion-R}
Q_{n-1}=\sum\mu^n Q
\end{equation}
for all $n\ge 1$.
\item Let $\cE$ be a \zec\ collection  for $\OI$, and let
$\bbeta^n$ for $n\ge 1$ be its maximal \cf s of order $n$.
Then, $Q$ is a decreasing fundamental sequence for the $\cE$-interval $\OI$
if and only if $Q$ satisfies (\ref{eq:full-recursion-R}) for all $n\ge 1$ where $\mu^n=\bbeta^n$ for each integer $n\ge 1$.

\end{enumerate}

\item (The weak converse)
Let $\cL$ be the \zec\ collection of   \cf s for $\OI$ determined by a \zml\
$L=(e_1,\dots,e_N)$, and let $\ome$ be the (only) positive real zero of the polynomial
$e_Nx^N +\cdots + e_1x - 1$.
Then, $\OI$ is an $\cL$-set of numbers, and $\OI$ has one and only one decreasing
fundamental sequence $Q$ given by $Q_k = \ome^k$ for all $k\ge 1$.

\end{enumerate}

\end{theorem}

Recall from the introduction the question on the weak converse for the $\cL$-set of numbers
$\OI$ where $L=(1,1)$.
Theorem \ref{thm:real}, Part 1 (b)
provides an affirmative answer to the question, and it is
given by $Q_k=\ome^k$ for all $k\ge 1$ where $\ome$ is the reciprocal of the
golden ratio.
In fact, the weak converse for the $\cL$-set of numbers $\OI$ is proved in
\cite{daykin3} for the case of
$L=(1,0,0,\dots,0,1)$ where $e_2=\cdots = e_{N-1}=0$.
Our idea is similar to \cite{daykin3},
and our proof relies on \cite[Theorem 1]{daykin3}, but also we improve
\cite[Theorem 1]{daykin3} in Proposition \ref{prop:dominant} below.
We prove the proposition in Section \ref{sec:proofs} in a more general setting.

As demonstrated in Theorem \ref{thm:neg-coeff},
if $\bbeta^n$ for $n\ge 1$ have common \zquote{tails},
it admits a short recursion.
Let $\bbeta^n:= 2\beta^n
+ \sum_{k=n+1}^\infty \beta^k$, which is an analogy of Example \ref{exm:neg-coeff} for $\OI$.
Then, if exists, a decreasing fundamental sequence must satisfy
$Q_n = 3 Q_{n+1} - Q_{n+2}$.
Its characteristic polynomial is $x^2-3x+1$, and
if we impose the decreasing property to Binet's Formula for $Q_n$, we
find that $Q_k = \ome^k$ for
$k\ge 1$ where $\ome=\tfrac12(3-\sqrt{5})$ must be the case.
Since $Q$ satisfies (\ref{eq:full-recursion-R}), by Theorem \ref{thm:real}, Part 1 (b), it is a decreasing fundamental sequence, and in particular,
it is the only decreasing fundamental sequence.
For Theorem \ref{thm:real}, Part 2 originally we assumed the condition $e_1\le \cdots\le e_N$ to
obtain certain inequalities between the values of $\sum \ep Q$, which seemed necessary to derive the equality (\ref{eq:full-recursion-R}),
but later we could derive the inequality without the assumption.
The increasing condition on $e_k$ reminds us of the Parry condition; see
\cite{miller:minimality}, \cite{frougny}, and
\cite{parry}.
As noted in \cite{miller:minimality},
prior to the work of \cite{mw},
the \Lz's theorem was only known in the case where the \zml\ satisfies the Parry condition.

For Theorem
\ref{thm:real}, Part 1 (a),
we use the greedy algorithm that is similar to Algorithm \ref{alg:E} to find \cf s $\mu^n$ for $n\ge 1$.
Given an index $n\ge 0$, recursively define $e_k$ for $k=1,2,3,\dots$ to be the largest non-negative integer such that $
Q_n - (e_1 Q_{n+1} +\cdots + e_{k-1} Q_{n+k-1})> e_k Q_{n+k} $; the strict inequality allows us to find infinitely many indices $k$ with $e_k\ge 1$.
Then, $\mu^{n+1} =\sum_{k=1}^\infty e_k \beta^{n+k}$, and
$Q$ satisfies the equality (\ref{eq:full-recursion-R}).

\begin{example}\label{exm:R-greedy}\rm
Let $Q$ be a sequence given by $Q_k=1/(k+1)$ for all $k\ge 0$.
Then,
\begin{gather}
Q_{n-1}=Q_{n} + Q_{n^2+n-1}\quad\text{ for $n\ge 1$},
\label{eq:Qn-id}
\end{gather}
and
if we repeatedly apply the recurrence to the last term of the RHS of (\ref{eq:Qn-id}), we
obtain $Q_{n-1} = Q_{m_1} + Q_{m_2} + Q_{m_3}+\cdots$ where $m_1=n$ and
$m_{k+1} = m_k(m_k+1)$ for $k\ge 1$.
\iffalse
if we apply the recurrence to $Q_{n^2+n-1}$, then
$Q_{n-1}=Q_{n} + Q_{n^2+n}+Q_{t^2+t-1}$ where
$t=n^2+n$, i.e., $Q_{n-1}=Q_{n } + Q_{n(n+1)}+Q_{t^2+t-1}$.
If we continue unfolding the last terms using the recurrence, we find
$Q_{n-1} = Q_{m_1} + Q_{m_2} + Q_{m_3}+\cdots$ where $m_1=n$ and
$m_{k+1} = m_k(m_k+1)$ for $k\ge 1$.
\fi
Let $\cE$ be the \zec\ collection for $\OI$
determined by $\bbeta^{n} =\sum_{k=1}^\infty \beta^{m_k}$ for $n\ge 1$, e.g.,
$\bbeta^3 = \beta^3 + \beta^{12} + \beta^{156}+\cdots$.
Then,
$Q_{n-1} = \sum \bbeta^n Q$ for $n\ge 1$, and
hence, by Theorem \ref{thm:real}, Part 1 (b),
$Q$ is a decreasing fundamental sequence for the $\cE$-set of numbers $\OI$.
Moreover, if we use the identity (\ref{eq:Qn-id}) to apply the greedy algorithm described above that requires a strict inequality,
we find that the expression $Q_{n-1} = \sum \bbeta^n Q$ for $n\ge 1$
coincides with the one obtained by the greedy algorithm.
Thus, the expansion of each real number obtained by applying the greedy algorithm with $Q$ makes a unique $\cE$-expansion.
For example, recall the well-known expansion
$\frac\pi8=\sum_{k=0}^\infty \frac1{(4k+1)(4k+3)}=\frac13+
\frac1{35}+\frac1{63}+\cdots+=Q_2 + Q_{34} + Q_{62}+\cdots$
, but since $34\cdot 35 =1190$, it is not \Ez.
The greedy algorithm yields
$$\frac \pi8\approx Q_2 + Q_{16} + Q_{1844} + Q_{4683104} +\cdots
=\frac13 +\frac1{17}+\frac1{3^2\cdot 5 \cdot 41}
+\frac1{3^2\cdot 5 \cdot 7 \cdot 14867}+\cdots
$$
\iffalse
$$\frac1\pi \approx Q_3 + Q_{14} + Q_{608} + Q_{845028} +\cdots
$$
\fi
where $\beta^2 + 0\cdot \beta^{6}$ is the first non-zero proper \Ebb,
and this is the only way of expressing $ \pi/8$ under the \Ez\ condition.
The process of finding non-zero terms of $Q$ is similar to that of continued fraction expansions, i.e., $y:=\frac \pi8-(\frac13+\frac1{17})$ and
$ 1/y \approx 1844.27$.
In fact, the \zec\ collection $\cE$ enjoys the finite expansions of all rational numbers in $\OI$ as in the continued fraction expansion.
\end{example}
Let us introduce our results for $p$-adic integers.
A sequence $Q$ in $\zz_p$ is {\it decreasing}
if $\abs{Q_k}_p > \abs{Q_{k+1}}_p$.
\begin{theorem}\label{thm:p-adic}
Let $p$ be a prime number.

\begin{enumerate}
\item (\zec's Theorem for $p$-adic integers)
Let $\cE$ be an arbitrary collection of \cf s $\ep$ such that
$\ep_n <p$ for all $n\ge 1$.
If $Q$ is a decreasing sequence in $\zz_p$, then
$Q$ has the unique $\cE$-representation property.

\item (The weak converse)
Let $\cE_0$ be a \zec\ collection   for positive integers, and let $\cE$ be the completion
of $\cE_0$.
If $\ep_n\le \min\set{\sqrt p, (p-1)/2}$ for all $n\ge 1$ and all $\ep\in\cE$, then
each decreasing $\cE$-subset of $\zz_p$ has a unique decreasing fundamental sequence.
\end{enumerate}

\end{theorem}

Let $p=41$.
Then, the golden ratio $\phi$ is defined in the $p$-adic integers $\zz_p$, and let us consider the sequence $Q$ in $\zz_p$ given by $Q_k = (\phi^k + 3\bar\phi^k) p^{k-1}$.
If we consider the \zml\ $L=(1,1)$ as for the Fibonacci sequence,
by Theorem \ref{thm:p-adic},
$Q$ is the only decreasing fundamental sequence for the $\cL$-subset $X_Q$, and this can be considered a $p$-adic analogue of the \fib\ sequence in terms of the weak converse.
It satisfies the recurrence $Q_{n+2} = pQ_{n+1} + p^2 Q_n$ for all $n\ge 1$, but $X_Q\ne \zz_p$.
The sequences $Z$ defined by the classical \fib\ recurrence $Z_{n+2} = Z_{n+1} + Z_n$ make some $\sum \ep Z$ divergent in $\zz_p$.

The weak converse fails for the $\cL$-set of numbers $\zz_p$
if the values of \cf s are too large, so it does not satisfy the condition of Theorem \ref{thm:p-adic}, Part 2.
Let $Q$ be the sequence given by $Q_k=p^{k-1}$ for $k\ge 1$, and
let $L=(p-1,p)$.
The standard $p$-adic expansion of $\zz_p $ makes an
$\cL$-expansion $\sum \ep Q$, and it has the \Lunique\
by Theorem \ref{thm:p-adic}, Part 1.
However, if $Z_k = \big( p(p-1) \big)^{k-1}$, then
$\ord_p(Z_k)=k-1$ implies that
each $p$-adic integer is equal to $\sum\ep Z$
for a unique \clz\ \cf\ $\ep$.
Thus, there are two distinct decreasing fundamental sequences
for $\zz_p $, and hence, the weak converse fails for $\zz_p$ under that \clz-condition.

In proving Theorem \ref{thm:real}, Part 2,
we establish the following result under a more general condition on the coefficients.
The result is also found in \cite[Theorem 12.2]{ostrowski}, but
our proof is more detailed, and we introduce the proof in Section \ref{sec:real}.
 
If a polynomial $f(z)$ over the complex numbers
has a unique complex root with largest modulus,
it is called a {\it dominant polynomial}.
\begin{prop}\label{prop:dominant}
Let $f(z)=a_n z^n -a_{n-1}z^{n-1}-\cdots - a_1 z - a_0$ be a polynomial in $\real[z]$
such that
$a_k\ge 0$ for all $0\le k\le n$ and $a_na_0 \ne 0$.
If there are indices $m$ and $\ell$ such that $1 \le m<\ell\le n-1$, $\gcd(m,\ell)=1$, and $a_m a_\ell \ne 0$,
then $f(z)$ is a dominant polynomial.
\end{prop}

By a classical theorem of
\cite{brauer},
if $a_{n-1}\ge \cdots \ge a_0$,
then the unique positive zero of the polynomial $f(z)$ in Proposition
\ref{prop:dominant}
is a Pisot number, and hence, it is a dominant polynomial.
As proved in \cite{dubickas},
most polynomials are dominant, but
proving that a certain class of polynomials are dominant in general
may require some work.
Various constructions and tests for dominant polynomials are introduced in
\cite{dubickas}.

\section{Examples of non-standard \zec\ conditions}\label{sec:non-standard}

Recall that by Theorem \ref{thm:integers}, given an increasing sequence
in $\nat$, there is a \zec\ condition under which the sequence is
a fundamental sequence for $\nat$, and it can be constructed using a greedy algorithm.
In this section, we introduce some non-standard examples for which
the \zec\ conditions are more concrete than the one abstractly given by a greedy algorithm.

\subsection{Fixed blocks}

Recall the \zec\ collection defined in Example \ref{exm:zero-coeff}.
The \immpred s  are given by $\hbeta^n=
\sum_{k=1}^{n-1} \beta^k=(1,1,\dots,1,1,\bar 0)$ for $n\not\equiv 0\mod 3$,
and $\hbeta^{3n}=\sum_{k=1}^{3(n-1)} \beta^k
+\beta^{3n-1}=(\bar 1,0,1,\bar 0)$.
The unique increasing fundamental sequence
$Q$ is given by $Q_n := Q_{n-1}+\cdots + Q_2 + 2Q_1$ for $n\not\equiv 0\mod 3$ where $(Q_1,Q_2)=(1,2)$, $Q_{3n} = Q_{3n-1}
+ Q_{3n-3}+\cdots +Q_2 + 2Q_1 $ for all $n\ge 2$, and $Q_3=3$.
Let us show that these recursions reduce to $Q_n = 7 Q_{n-3}$ for $n\ge 4$
with $(Q_1,Q_2,Q_3)=(1,2,3)$.

Notice that
the values of $\cE$-expansions supported on the indices
$\set{3n-2,3n-1,3n}$ are obtained by adding $Q_{3n-2}$ each time as follows:
\begin{gather}
0<Q_{3n-2}<Q_{3n-1}< Q_{3n}
<Q_{3n-2} +Q_{3n}<Q_{3n-1} +Q_{3n}
<Q_{3n-2} +Q_{3n-1} +Q_{3n}< Q_{3n+1}.
	 \label{eq:listing}
\end{gather}
Hence, $7Q_{3n-2} = Q_{3n+1}$.
Also notice that as we keep adding $Q_{3n-2}$,
we arrive $Q_{3n-1}$ and $Q_{3n}$, and
we have $Q_{3n-1} =2 Q_{3n-2}$ and $Q_{3n}=3 Q_{3n-2}$.
Thus,
\begin{align*}
Q_{3n-1} =2 Q_{3n-2},\ n\ge 1 &\implies
Q_{3n+2}=2 Q_{3n+1}=7(2Q_{3n-2})=7 Q_{3n-1},\\
Q_{3n}=3 Q_{3n-2},\ n\ge 1 &\implies Q_{3n+3}
=3 Q_{3n+1} =7(3Q_{3n-2})= 7 Q_{3n}.
\end{align*}
Let us introduce a different perspective,
from which it is far easier to see the simple formula of $Q_n$.
First notice that $(1,2,3)$ is the only increasing sequence of three positive integers
that uniquely represents the first consecutive positive integers $\set{1,2,\dots,M}$
under the \Ez\ condition where $M$ turns
out to be $6$.
Given any positive integers $n$, let $n=\sum_{r=0}^\infty a_r 7^r$ be
the base-$7$ expansion where $0\le a_r \le 6$, and
write $n=\sum_{r=0}^\infty (\mu_{r1} +2 \mu_{r2} + 3\mu_{r3})7^r$
where $(\mu_{r1}, \mu_{r2}, \mu_{r3})$ is the unique one of the seven blocks such that $a_r =\mu_{r1} +2 \mu_{r2} + 3\mu_{r3}$. Thus, they make an $\cE$-expansion.

In general given a positive integer $N\ge 1$,
we may consider a random list of
\immpred s $\hbeta^k$ for $2\le k\le N+1$.
They determine an increasing finite sequence
$(Q_1,\dots,Q_N)$ where $Q_1=1$ and $Q_n = Q_1+\sum\hbeta^n Q$
for $2\le n\le N$, and let $B$ be the list of blocks of length $N$
that were used in representing each integer from $1$ to $\sum \hbeta^{N+1} Q$.
Then, the collection $\cE$ generated by concatenating blocks in $B$
is \zec.
Consider the sequence $Q$ given by $Q_n = b Q_{n-N}$ where $b=Q_{N+1}$ and the initial values $(Q_1,\dots,Q_N)$ that were the ones determined earlier.
Then, as argued in the earlier example where $b=7$ and $N=3$,
it is easy to see via the base $b$-expansions that
$Q$ is an increasing fundamental sequence for $\nat$.
By Theorem \ref{thm:integers},
under the \Ez\ condition,
$\nat$ has a unique increasing
fundamental sequence, and hence,
the system of equalities $Q_{n+1}=Q_1 + \sum\hbeta^n Q$
must imply the short recursion $Q_n = b Q_{n-N}$ for all $n\ge N+1$.
However, we can also use the method of
listing expressions using
all blocks as in (\ref{eq:listing}), i.e.,
\begin{gather}
Q_{nN+1} <\cdots
<\sum_{k=nN+1}^{(n+1)N} \hbeta^{N +1}_{k-nN} Q_k < Q_{(n+1)N +1},\label{eq:listing-1}\\
\intertext{
and as we list each block,
we obtain}
Q_{nN+k} = c_k Q_{nN+1}\quad \text{for $1\le k\le N$}
\label{eq:listing-2}
\end{gather}
where $c_k$ is independent of $n$.
This proves that $Q_{n+N} = b Q_n$ for all $n\ge 1$
where $b=Q_{N+1}$.

Let us consider \zec\ collections constructed with these fixed blocks for
the unit interval $\OI$.
Let $B$ be the descendingly-ordered list of  blocks
$\set{(a_1,a_2,a_3)\in\zz^3 : 0\le a_k \le 1,\ k=1,2,3}- \set{(0,1,1)}$,
and let us declare the maximal \cf s of order $n$ by concatenating these seven blocks.
For $n\ge 0$,
define
\begin{gather}
\begin{aligned}
\bbeta^{3n+1} &:=\beta^{3n+1} + \beta^{3n+2} + \beta^{3n+3}
+\sum_{k=3n+4}^\infty \beta^k\\
\bbeta^{3n+2} &:=\phantom{\beta^{3n+1} +} \beta^{3n+2} \phantom{ + \beta^{3n+3} }
+\sum_{k=3n+4}^\infty \beta^k\\
\bbeta^{3n+3} &:= \phantom{\beta^{3n+1} + \beta^{3n+2} +} \beta^{3n+3}
+\sum_{k=3n+4}^\infty \beta^k,
\end{aligned}
\label{eq:block-eq}
\end{gather}
and let
$\cE$ be the \zec\ collection determined by $\bbeta^n$ for $n\ge 1$.
If exists, by Theorem \ref{thm:real}, a decreasing fundamental sequence $Q$ must satisfy the following
for $n\ge 0$
since $Q_{3n+3}=\sum_{k=4}^\infty Q_{3n+k}$:
\begin{equation}\label{eq:system}
\begin{aligned}
Q_{3n}&
	 =Q_{3n+1}+Q_{3n+2}+2Q_{3n+3}\\
Q_{3n +1}&
	 =\phantom{Q_{3n+1}+}Q_{3n+2}+ Q_{3n+3}\\
Q_{3n+2 }&
	 = \phantom{Q_{3n+1}+Q_{3n+2}+} 2Q_{3n+3}.
	 \end{aligned}
	 \end{equation}
	 The situation is similar to the case of positive integers.
	 Recall the subcollection $\cE^{3n+3}$, which is finite, from Definition \ref{def:immsucc-R}.
The following seven values are obtained from using \cf s in $\cE^{3n+3}$, and they are equal to consecutive integer multiples of
$Q_{3n+3}$:
$$
0<
Q_{3n+3} < Q_{3n+2} <Q_{3n+1}
<\cdots <Q_{3n+3} + Q_{3n+2} + Q_{3n+1} < Q_{3n},$$
and $Q_{3n+2}=2Q_{3n+3}$ and $Q_{3n+1} = 3 Q_{3n+3}$ for all $n\ge 1$. Hence, $Q_n = 7 Q_{n+3}$ for all $n\ge 1$.
If we use the same idea of decomposing coefficients of the base $1/7$-expansion, it is easy to see that
$(Q_1,Q_2,Q_3)=(3/7,2/7,1/7)$ and $Q_n = 7 Q_{n+3}$ for all $n\ge 1$ define a decreasing fundamental sequence for $\OI$ under the \Ez\ condition.

Let us consider the problem of the weak converse under the \Ez\ condition,
i.e., the problem of determining whether this is the only decreasing fundamental sequence.
Since the characteristic polynomial for the recursion is
$7x^3-1$, and its zeros have the moduli equal to each other,
\cite[Theorem 1]{daykin3} does not allow us to conclude that
a decreasing fundamental sequence is given by
$ a /\sqrt[3]{7}^n$ for $n\ge 1$.
In fact,
there are complex numbers $\al_k$ for $k=1,2,3$ such that the terms of the above fundamental sequence $Q$ are given by
$Q_n = \al_1 \frac1{\sqrt[3]{7^n}} + \al_2 \frac{e^{2\pi in/3}}{\sqrt[3]{7^n}}+ \al_3 \frac{e^{2\pi in/3}}{\sqrt[3]{7^n}}$
for all $n\ge 0$ where
$\al_1 \approx .96$, $\al_2\approx .02 + .07 i$, and $\al_3\approx .02 - .07 i$.
It is possible to use Binet's formula and consider other possibilities
for $\al_k$ to prove that $Q$ is the only decreasing fundamental sequence, but also we can specialize the system (\ref{eq:system}) with $n=0$ to obtain
$(Q_1,Q_2,Q_3)=(3/7,2/7,1/7)$.
This proves that $Q$ is the only decreasing fundamental sequence under
the \Ez\ condition, i.e.,
the weak converse holds for the $\cE$-interval $\OI$, but
as in the integer case Example \ref{exm:converse-fail}, the full converse fails since
the initial values $(Q_1,Q_2,Q_3)=(3/7,1/7,2/7)$ makes a fundamental sequence as well.

The setup (\ref{eq:block-eq}) can be generalized as follows.
Let $N$ be a positive integer, and given integer $k=1,2,\dots,N$,
let
$a_k$ be a positive integers and $b_{kj}$ be non-negative integers
for $j=k+1,\dots,N$.
Let $\cE$ be the \zec\ collection   for $\OI$ determined by
maximal \cf s $\bbeta^m$ of order $m$ such that
$$\bbeta^{Nn+k}=a_k \beta^{Nn+k} + \sum_{j=k+1}^{N} b_{kj} \beta^{Nn+j}
+\bbeta^{N(n+1)+1}$$
for $k=1,\dots,N$ and $n\ge 0$.
Then, we have a corresponding system of linear equations for $Q_1,\dots,Q_N$
we obtain as in
(\ref{eq:system}) with $n=0$.
Its row reduction becomes a nonhomogeneuous system since $Q_0=1$, and the row-reduced system has no shifts in leading positions, i.e., it is nonsingular.
Thus, it is clear that, if exists, the first $N$ values of a decreasing fundamental sequence $Q$ are uniquely determined, and
it must be given by $Q_n = b Q_{n+N}$ according
to the principle introduced in (\ref{eq:listing-1}) and
(\ref{eq:listing-2}).
By Theorem \ref{thm:real}, it is indeed a decreasing fundamental sequence for the $\cE$-interval $\OI$, and hence,
the weak converse holds.

\subsection{Linear recurrence with negative coefficients}
\label{sec:neg-coeff}

Recall Example \ref{exm:neg-coeff}.
The recurrence was obtained by considering
$\hbeta^n = \sum_{k=1}^{n-2} \beta^k + 2\beta^{n-1}$
where $
\sum_{k=1}^{n-2} \beta^k$ allows $Q_n $ and $Q_{n-1} $ to have a common tail, i.e.,
$Q_n = 2Q_{n-1} + \sum_{k=1}^{n-2} Q_k$
and $Q_{n-1}= 2Q_{n-2} + \sum_{k=1}^{n-3} Q_k$
have a common tail $\sum_{k=1}^{n-3} Q_k$, and
this allows us to derive a short recursion
$Q_n = 3Q_{n-1} - Q_{n-2}$.
We have
the following general result.
\begin{prop}\label{prop:neg-coeff}
Let $N$ and $b$ be non-negative integers such that $b>0$ if $N=0$,
and let $e_k$ for $1\le k\le N$ be non-negative integers such that $e_1\ge 1$ if $N>0$.
Let $\cE$ be a \zec\ collection for positive integers
determined by \immpred s $\hbeta^n=
	 \sum_{k=1}^{n-1} e_k \beta^{n-k}$ for $2\le n\le N+1$, and
for $n\ge N+2$,
$\hbeta^n=\sum_{k=1}^{n-N-1} b \beta^k
	+ \sum_{k=1}^N e_k \beta^{n-k}$ (where coefficients $e_k$ are independent of $n$).
	Let $Q$ be
	the increasing fundamental sequence for the $\cE$-set of numbers $\nat$.
Then, for $n\ge N+2$,
\begin{equation}
Q_n =(e_1+1) Q_{n-1}
+ \sum_{k=2}^N (e_k-e_{k-1})Q_{n-k}
+ (b-e_N) Q_{n-N-1}. \label{eq:neg-coeff}
\end{equation}

\end{prop}

\begin{proof}
If $n\ge N+2$,
\begin{align*}
&
\begin{aligned} Q_n &=Q_1+ \sum_{k=1}^{n-N-1} b \beta^k Q_k
	+ \sum_{k=1}^N e_k Q_{n-k} \\
	& =Q_1+ \sum_{k=1}^{n-N-2} b \beta^k Q_k+ b Q_{n-N-1}
	+ \sum_{k=1}^N e_k Q_{n-k},\end{aligned} \\
	& Q_{n-1}  =Q_1+ \sum_{k=1}^{n-N-2} b \beta^k Q_k
	+ \sum_{k=1}^N e_k Q_{n-k-1} \\
	\implies\quad
& Q_n - Q_{n-1}  =e_1 Q_{n-1}
+ \sum_{k=2}^N (e_k-e_{k-1})Q_{n-k}
+ (b-e_N) Q_{n-N-1}\\
\implies\quad
& Q_n  =(e_1+1) Q_{n-1}
+ \sum_{k=2}^N (e_k-e_{k-1})Q_{n-k}
+ (b-e_N) Q_{n-N-1}.
\end{align*}

\end{proof}
\noindent
The structure of the coefficients in (\ref{eq:neg-coeff})
immediately implies Theorem \ref{thm:neg-coeff}, and we leave the proof to the reader.

\subsection{Linear recurrence with non-constant coefficients}
Let us introduce \immpred s $\hbeta^n$ that are
written in terms of polynomials with
positive coefficients.
It is still based on the same idea of using a common tail in \immpred s as in Section \ref{sec:neg-coeff}.
Let $\tau=\sum_{k=1}^\infty k \beta^k$ be a \cf,
and let $\tau^n$ denote $\resab{[1,n]}(\tau)$ for $n\ge 1$.
Let
$\cE$ be the \zec\ collection for positive integers with 
$\hbeta^n = \tau^{n-2} + n \beta^{n-1}=(1,2,3,\dots,n-2,n,\bar0)$ for all $n\ge 2$.
Let $Q$ be the increasing fundamental sequence for the $\cE$-set of numbers $\nat$, so that it is given by
$Q_n = Q_1 + \sum\hbeta^n Q$.
Then, using the same idea introduced in Section \ref{sec:neg-coeff},
we find
$ Q_n = (n+1) Q_{n-1}- Q_{n-2} $ for all $n\ge 3$.
For the same \cf\ $\tau$,
if we instead consider \immpred s
$\hbeta^n =\tau^{n-3} + 2n \beta^{n-2}+3n \beta^{n-1}$,
then
the recursion turns out to be
$Q_n=(3n+1) Q_{n-1} -(n-3) Q_{n-2} - (n+1) Q_{n-3}$
for all $n\ge 4$.

Recall from Section \ref{sec:results} the example of a fundamental sequence $Q$ given by
$Q_{n+2} = (n+1) Q_{n+1} + Q_n$ for all $n\ge 1$ and
$Q=(1,2,5,17,\dots)$.
This recursion is obtained from
$\hbeta^n = \sum_{k=0}^{\lfloor (n-1)/2\rfloor}(n-1-2k) \beta^{n-1-2k}=(\cdots,0,n-5,0,n-3,0,n-1,\bar 0)$ for all $n\ge 2$, and
if we use the same idea demonstrated in this section
and the expression
$\hbeta^n = (n-1) \beta^{n-1} + \hbeta^{n-2}$ for all $n\ge 4$,
i.e.,
\begin{align*}
Q_{n} &= (n-1) Q_{n-1} +(n-3) Q_{n-3} +\cdots\\
Q_{n-2} &= (n-3) Q_{n-3} +(n-5) Q_{n-5} +\cdots
\end{align*}
then we obtain $Q_{n+2} = (n+1) Q_{n+1} + Q_n$ for all $n\ge 1$.

Conversely, if we have
$$Q_{n+3}= f_1(n+2) Q_{n+2} + f_2(n+1) Q_{n+1} + f_3(n) Q_n$$ for all $n\ge 1$, for example, where $f_k(n)$ for $k=1,2,3$ are non-zero polynomials with non-negative integer coefficients, we may recursively define \immpred s as follows:
$\hbeta^n = f_1(n-1) \beta^{n-1} + f_2(n-2) \beta^{n-2}
+(f_3(n-3) - 1) \beta^{n-3} + \hbeta^{n-3}$ for all $n\ge 5$,
$\hbeta^4 = f_1(3) \beta^3+f_2(2) \beta^2+(f_3(1)-1) \beta^1$,
$\hbeta^3 = f_1(2) \beta^2 + f_2(1) \beta^1$, and $\hbeta^2 =f_1(1) \beta^1$.
For example,
$$
\hbeta^6
=f_1(5)\beta^5 + f_2(4)\beta^4 +( f_3(3)-1) \beta^3
+ f_1(2)\beta^2 + f_2(1)\beta^1.$$

\section{Examples of $\cE$-\zec\ subsets }\label{sec:examples}

Recall from Section \ref{sec:D-R} that $a\nat$ is an $\cL$-subset of $\nat$ if $L=(1,1,1)$ and $a$ is a positive integer.
In fact, $a\nat$ is an $\cE$-subset of $\nat$ for all \zec\ collections $\cE$
since $\nat$ is an $\cE$-set of numbers for all \zec\ collections $\cE$ by Theorem \ref{thm:integers}.
It turns out that other congruence classes are essentially not an $\cE$-subset for any
  \zec\ collection  $\cE$ that contains more than the basis \cf s $\beta^n$ and the zero \cf.
  We shall call such a collection a nontrivial \zec\ collection.
\begin{prop}
Given an integer $a>1$ and an integer $0< j <a$ such that $\gcd(j,a)=1$, the subset $j+a\nat$ is not an $\cE$-subset if $\cE$ is not trivial.
\end{prop}

\begin{proof}
Let $\cE$ be a nontrivial \zec\ collection for $\nat$, and suppose
that $j+a\nat$ is an $\cE$-subset of $\nat$, so that it has a fundamental sequence $Q$.
Since $\cE$ is not trivial,
there must be a \cf\ $\mu$ whose \immsucc\ is given by $\td\mu=\beta^1 + \mu$.
Then, $\set{Q_1,\sum \mu Q}\subset j+a \nat$ implies
$\sum \td \mu Q = Q_1 + \sum \mu Q\equiv j+j \mod a$.
Since $\sum \td \mu Q\in j + a\nat$, it follows that $2j \equiv j \mod a$, and
since $\gcd(j,a)=1$, it implies that $2\equiv 1 \mod a$, which is a contradiction. This proves the proposition.

\end{proof}
In fact, if $\cE$ is trivial, then we may define $Q_k:=j + ak$ and $j+a\nat = X_Q$. The only $\cE$-expansions are $Q_k$ since $\cE=\set{\beta^k : k\ge 1}
\cup\set{0}$.
The case of $\gcd(j,a)>1$ is reduced $j+a\nat = d(j_0 + a_0\nat)$
where $d=\gcd(j,a)$.
If $a_0>1$, then the proposition implies that $\cE$ must be trivial,
but if $a_0=1$, then the answer is different.
\begin{prop}
Let $j$ be a positive integer.
Then, $j + \nat$ is an $\cE$-subset of $\nat$ for a nontrivial \zec\ collection $\cE$.
\end{prop}

\begin{proof}
Let $\hbeta^{n}=\beta^{n-1}$ for $2\le n\le j+1$, and
$\hbeta^{n}= \beta^j + \beta^{n-1}$ for $n\ge j+2$.
Let $\cE$ be the \zec\ collection determined by \immpred s $\hbeta^n$ for $n\ge 2$.
Let $Q$ be an increasing sequence given by
$Q_k=j+k$ for $1\le k\le j$, and
$Q_n=j(n-j)$ for $n\ge j+1$.
Then, $(Q_1,\dots,Q_j)$ makes a complete residue system mod $j$, and
the terms $Q_k$ for $k\ge j+1$ represent multiples of $j$.
Hence,
given a number $j+x\in j+\nat$, by the long division algorithm,
there are unique integers $k$ and $q$ such that
$x = k + jq$ and $1\le k\le j$.
If $q=0$, i.e., $1\le x\le j$, then $j+x=Q_k$, and
if $q>0$, i.e., $x> j$, then
$j+x=Q_k+Q_n$ where $n:=q+j\ge j+1$.

First notice that the following are typical sequences of \immsucc s:
\begin{gather*}
\beta^1 \aord \beta^2\aord \cdots \aord \beta^j,\\
\intertext{Given $n\ge j+1$,}
\beta^n \aord
\beta^1 +\beta^n \aord \beta^2 +\beta^n
\aord \cdots \aord \beta^j +\beta^n\aord \beta^{n+1}.
\end{gather*}
Then,
given a number $j+x\in j + \nat$, by the earlier decomposition using the long division algorithm, there is a unique \cf\ $\mu\in \cE$ such that
$j+x = \sum \mu Q$. This proves that $j+\nat \subset X_Q$.
Since all non-zero \cf s $\mu \in \cE$ are either $\beta^k$ or $\beta^k + \beta^n$ for $1\le k\le j$ and $n\ge j+1$,
clearly $\sum \mu Q \in j + \nat$, i.e., $X_Q\subset j + \nat$.

\end{proof}

Let us introduce another method of generating $\cE$-subsets of $\nat$,
and it is inspired by the shifting function introduced in \cite{kimberling}.
Let $Z$ be the sequence given by $Z_k:=Q_{k+1}$ where
$Q$ is a fundamental sequence for the $\cE$-set of numbers $\nat$,
and let $\psi_Q : \nat \to \nat$ be the function defined by
$\psi_Q(\sum \ep Q) = \sum \ep Z$.
The unique representation property of $Q$ under the \Ez\ condition implies that the function is well-defined.
For example, if $L=(1,1)$ is a \zml\ and $Q$ is the \fib\ sequence under the \Lz\ condition,
then $\psi_Q(100)=\psi(F_{10}+F_5+F_3) = F_{11}+F_6+F_4 = 162$, and
in general, $\psi_Q(n) =\phi n + O(1)$ where $\phi$ is the golden ratio.
Clearly the subset $\psi_Q(\nat)$ is an $\cL$-subset, and
$\psi_Q(\nat) = \set{\sum \ep Q : \ep\in\cL\ \&\ \ep_1=0}$, which
has a positive proportion of $\nat$ in the following sense:
$
\frac1x\, \#\set{m \le x : m \in \psi_Q(\nat)}
\to \phi -1 $ as $x\to\infty$.
If $\cE$ is an arbitrary \zec\ collection,
there is no difficulty in defining the shifting function $\psi_Q : \nat \to \nat$ where $Q$ is a fundamental
sequence for the $\cE$-set of numbers $\nat$,
and $\psi_Q(\nat)$ and $\psi_Q(\cdots(\psi_Q(\nat)))$ generate a good deal of interesting $\cE$-subsets of $\nat$.

Another interesting way of constructing $\cE$-subsets are as follows.
Let $\wt\cE$ be a \zec\ collection,
and suppose that it has a proper \zec\ subcollection $\cE$.
If $Q$ is a decreasing fundamental sequence
for the $\wt\cE$-set of numbers $R$, then
we may investigate the subset $X_Q^{\cE}$.
Recall from Section \ref{sec:D-R} the sequence $Q$ given by $Q_k:= 2^{k-1}$ for $k\ge 1$, which is an increasing fundamental sequence under the $\wt\cL$-\zec\ condition determined by a \zml\ $\wt L=(1,2)$, and
the subset $X_Q^{\cL}$ where $\cL$ is the \zec\ collection determined by
$L=(1,1)$.
The $\cL$-subset $X_Q^{\cL}$ contains nearly zero proportion of $\nat$
since $\#\set{ m \in X_Q^{\cL} : m \le 2^n} = F_{n+1}$ which implies
$\#\set{ m \in X_Q^{\cL} : m < x} = O(x^{\log_2(\phi)})$ where $\phi\approx 1.6$ is the golden ratio.

\newcommand{\wLz}{$\wt\cL$-\zec}
Let us consider \zec\ collections $\wt \cL$ and $\cL$
of  \cf s for $\OI$ determined by
\zml s $\wt L = (1,2)$ and $L=(1,1)$, respectively.
Then, the sequence $Q$ given by $Q_k=1/2^k$ for $k\ge 1$ is a decreasing
fundamental sequence for the $\wt\cL$-interval $\OI$.
As in the case of $\nat$, the subset generated by $Q$ under
the $\cL$-\zec\ condition is small in the following sense:

\begin{prop}
Let $Q$ be the sequence given by $Q_k=1/2^k$ for all $k\ge 1$,
and let $L=(1,1)$.
Then, the $\cL$-subset $X_Q$ is Lebesgue measurable, and
it has Lebesgue measure zero.
\end{prop}

\begin{proof}
Let $\wt L=(1,2)$ be a \zml, and let $\wt \cL$ be the \zec\ collection
determined by $\wt L$.
Let $S$ be the subset of $X_Q$ consisting of $\sum\ep Q$ such that
$\ep\in \cL$ is non-zero, and has finite support.
Given a positive
integer $b$,
let $S_b$ be the subset of $S$ consisting of $\ep$
such that $\ep_b=1$ and $\ep_k=0$ for all $k>b$.
Then, $S$ is the disjoint union of $S_b$ where $b$ varies over the
positive integers.

Given $\gamma \in S$, find the index $b$ such that
$\gamma\in S_b$.
Let $\gamma_0:=\gamma+1/2^{b+1}$, and
\begin{equation*}
J_\gamma:=(\gamma_0,\gamma_0+1/2^{b+1})
=\set{ \gamma_0 + \sum \delta Q \in \OI : \delta\in \wt\cL,\ %
b+1<\ord(\delta)<\infty}.
\end{equation*}
Then,
the binary expansions of a real number $x$ in the intervals $J_\gamma$
are in the form of $\sum_{k=1}^{b-1} \ep_k 2^{-k} + 2^{-b}+2^{-b-1}
+\sum_{k=b+2}^\infty \delta_k 2^{-k}$ where $\ep\in\cL$ and
$\delta\in\wt\cL$, which are non-zero.
Since $\ep\in\cL$, the sum $2^{-b}+2^{-b-1}$ is the first adjacent terms
in the expansion of $x$.
Conversely, if the binary expansion of a number in $J_\gamma$ is given,
we can identify $\gamma$ by finding the first two adjacent terms.
If $\gamma$ and $\gamma'$ are two numbers in $S$,
and $J_\gamma \cap J_{\gamma'}\ne \varnothing$,
then the first two adjacent terms of the (unique) binary expansion of a number in $J_\gamma \cap J_{\gamma'}$ determine $\gamma$, and hence,
$\gamma=\gamma'$.
Thus, $\bigcup_{\gamma\in S} J_\gamma$ forms a disjoint (countable) union of open intervals.
For all $\gamma\in S$, the left endpoint $\gamma_0$ of the interval $J_\gamma$ is not a member of $X_Q$, but the right endpoint may or may not be a member of $X_Q$.
Let $J_\gamma^0$ denote $[\gamma_0,\gamma_0+1/2^{b+1}]$
if the endpoint is not in $X_Q$, and $[\gamma_0,\gamma_0+1/2^{b+1})$
otherwise. Then since $S$ is countable,
the union $\bigcup_{\gamma\in S} J_\gamma^0$ is Lebesgue measurable,
and $\OI - X_Q = \bigcup_{\gamma\in S} J_\gamma^0$.
It is a straightforward induction exercise to show that given an integer $b\ge 1$, the \fib\ term
$F_{b-1}$ where $F_0=1$ is the number of \cf s in $S_b$.
\newcommand{\mes}{\mathrm{msr}}%
Let
$\mes$ denote the Lebesgue measure.
Then, it follows
\begin{align*}
\mes(\OI-X_Q)=\mes\left(\bigcup_{\gamma\in S} J_\gamma^0 \right)
&=\mes\left(\bigcup_{\gamma\in S} J_\gamma \right)
=\sum_{b=1}^\infty\ \sum_{\gamma\in S_b} \mes\left(J_\gamma \right) =
\sum_{b=1}^\infty \frac{F_{b-1}}{2^{b+1}} = 1.
\end{align*}
Therefore, $X_Q$ is Lebesgue measurable, and $\mes(X_Q)=0$.

\end{proof}

\iffalse
\begin{conjecture}
Given a finite list of positive integers $\cL$,
each Lebesgue measurable \clz\ subset of $\OI$ that is not an open interval has measure zero.

\end{conjecture}

Example: If $Q_k=1/2^k$ and $\cL=(1,1)$, then it seems that $X_Q$ has measure zero.
All the intervals $J_n^1:=[1/2^n + 1/2^{n+1},1/2^{n-1})$ for $n\ge 2$ are gone.
All the intervals $J_n^2:=[1/2^n + 1/2^{n+2}+1/2^{n+3},1/2^n + 1/2^{n+1})$ for $n\ge 2$ are gone.
\fi

\section{Proofs}\label{sec:proofs}
We prove the results that are introduced in Section \ref{sec:D-R}.   
The following lemma is fundamental
to the connection between the two descriptions of \zec\ collections 
introduced in Theorem \ref{thm:block-def-N}.

\begin{lemma}\label{lem:decomposition}
Let $\cE$ be a \zec\ collection   for positive integers,
and let $\delta = \al + \rho\in\cE$ where $\al$ and $\rho$ are \cf s
such that $\al\ne 0$ and $\rho \equiv 0 \reS [1,r)$ where $r:=\ord(\al) $.
Then,
there is an index 	$t$ such that $t>r$,
$\resab{[r,t)}(\rho)$ is a proper \Eb\ at index $t-1$,
and one of the successors of $\delta$ in the \lex\ order
is $\beta^{t} + \resab{[t,\infty)}(\rho)$.
\end{lemma}
\begin{proof}
Let $\delta^1:=\delta$, and let $\delta^k$ be the immediate successor of $\delta^{k-1}$ for $k=2,\dots$, which are obtained by adding $\beta^1$ or replacing a copy of
$\hbeta^u$ for some index $u$.
Let $S$ be the set of indices $k$ such that
$\delta^k \equiv \delta \reS (r,\infty)$.
Since $1\in S$, the set $S$ is non-empty, and
by Definition \ref{def:zec-N}, it is a finite set.
We denote by $K$ the largest index in $S$.
By the choice of $K$, we must have 
$\delta^{K+1}\ne \delta^K + \beta^1$,
because otherwise we would have $K+1\in S$.
Then, by Part 2 of Definition \ref{def:zec-N}, there is an index $t\ge 2$ such that
$\delta^K \equiv \hbeta^t \reS [1,t)$ and
$\delta^{K+1} = \beta^t + \resab{[t,\infty)}(\delta^K)$.
Let us claim that $t>r$.
If $t\le r$, then $\delta^{K+1}\equiv \delta^K \equiv
\delta \reS (r,\infty)$, which implies $K+1\in S$.

Write $\delta^K = \al^K + \rho$.
Since $\delta^K\equiv\delta\equiv \rho \reS (r,\infty)$, we have
$\al^K \equiv 0 \reS (r,\infty)$.
Notice that $\hbeta^t \equiv \al^K + \rho \reS [1,t)
\implies \hbeta^t \equiv \al^K + \rho \reS [r,t)$.
Since $\al^K \equiv 0 \reS (r,\infty)$, it follows that
$\hbeta^t \equiv \al^K_r \beta^r + \rho \reS [r,t)$, and
$\hbeta^t \equiv \al^K + \rho \equiv \rho \reS (r,t)$.
By the definition of the \lex\ order,
$\al^K_r \ge \al_r \ge 1$, and hence,
and $ \hbeta^t_r=\al^K_r + \rho_r>\rho_r$. Thus, $\rho \reS [r,t)$ is a proper \Eb\ at index $t-1$.
In particular, if $t=r+1$ and $\rho_r=0$, then
the proper \Eb\ at index $t-1$ is the zero \cf.
\end{proof}

\begin{cor}\label{cor:decomposition}
Let $\rho$ be the \cf\ defined in Lemma \ref{lem:decomposition}.
Then $\rho$ has a decomposition into proper \Eb s.
\end{cor}

\begin{proof}
Let $t_0:=r$.
By Lemma \ref{lem:decomposition},
there is an index $t_1>t_0$ such that $\zeta^1:=\resab{[t_0,t_1)}(\rho)$ is
a proper \Eb\ at index $t_1-1$.
Also since $\beta^{t_1} + \resab{[t_1,\infty)}(\rho)$ is a member of
$\cE$, the lemma with $\al=\beta^{t_1}$ again implies that
there is an index $t_2>t_1$ such that $\zeta^2:=\resab{[t_1,t_2)}(\rho)$ is
a proper \Eb\ at index $t_2-1$ where
$\hbeta^{t_2} + \resab{[t_2,\infty)}(\rho)\in \cE$.
This process continues and generates an infinite sequence
of \cf s $\beta^{t_k} + \resab{[t_k,\infty)}(\rho)$ in $\cE$ where
$\resab{[t_{k-1},t_k)}(\rho)$ is a proper \Eb\ at index $t_k-1$ for $k\ge 1$.
Then, the expression
$\rho=\sum_{k=1}^\infty \resab{[t_{k-1},t_k)}(\rho)$
serves as
a decomposition into proper \Eb s where all 
\Eb s with sufficiently large $k$ are the zero \cf\ (since $\rho$ has finite support).

\end{proof}

The following lemma is also fundamental to the connection
described in Theorem \ref{thm:block-def-N}.

\begin{lemma}\label{lem:E-unique}
Let $E$ be the \aorcoll\ collection of \cf s determined by $\set{\delta^n : n\ge 2}$  defined in Definition
\ref{def:generate-zec}. Then, the following are true:
\begin{enumerate}
\item Each \cf\ $\mu$ in $E$ has a unique $\delta$-block decomposition, and
the largest $\delta$-block in the decomposition is given at index $\ord(\mu)$.
\item
Each \cf\ in $E$ has a unique \immsucc\ in $E$, and it is given by the rule
described in Theorem \ref{thm:block-def-N}, Part 1 (b) where $\hbeta^n =\delta^n$ for each $n\ge 2$.  In particular, for each $n\ge 2$, $\delta^n$ is the \immpred\ of $\beta^n$, and the rule satisfies Definition \ref{def:zec-N}, Part 2.
\item Given $\mu\in E$, there are only finitely many \cf s in $E$ that are less than $\mu$.
\end{enumerate}
In particular, $E$ is a \zec\ collection for positive integers with 
$\hbeta^n =\delta^n$ for each $n\ge 2$.
\end{lemma}

\begin{proof}
Let us prove Part 3.
If $\mu' \aord \mu$ are two \cf s in $E$, then $\ord(\mu')\le \ord(\mu)$,
and $\mu'$ is written in terms of $\delta$-blocks  of order $\le \ord(\mu)$.
Since there are only finitely many $\delta$-blocks of order $\le \ord(\mu)$,
we prove Part 3.

Let us use the mathematical induction on the order of \cf s to prove Part 1.
The case of $\ord(\ep)=0$ is trivial.
Assume that the \cf s $\ep\in E$ with $\ord(\ep)=s \le M$ for some $M\ge 0$ have a unique $\delta$-block decomposition, and
let $\mu$ be a \cf\ in $E$ with $\ord(\mu) =M+1$.
Let $\mu= \sum_{m=1}^{A}\zeta^m=\sum_{m=1}^{B}\xi^m$ be two $\delta$-block decompositions with supports
$[i_m,n_m]$ and $[j_m, t_m]$, respectively, such that
$i_1=1$, $n_m+1=i_{m+1}$ for $1\le m\le A$, $\zeta^{A}\ne 0$, $j_1=1$, $t_m+1=j_{m+1}$ for $1\le m\le B$, $\zeta^{B}\ne 0$.
Since $\ord(\mu)=\ord(\zeta^A)=\ord(\xi^B)$,
 both $\zeta^{A}$ and $\xi^{B}$ are non-zero
$\delta$-blocks at index $c:=n_{A}=t_{B}$, and for convenience,
let $[a,c]$ and $[b,c]$ denote the supports of $\zeta^{A}$ and $\xi^{B}$, respectively.
If $b<a$, then
$\delta^{c+1}_a=\xi^{B}_a=\zeta^{A}_a <\delta^{c+1}_a$, which is a contradiction, and $a<b$ implies a similar contradiction.
Thus, we conclude $a=b$, and after cancelling $\zeta^{A}$ and $\xi^{B}$, by induction hypothesis, we prove the uniqueness property.
Given a unique $\delta$-block decomposition $\mu= \sum_{m=1}^{A}\zeta^m $
where $\ord(\mu)=\ord(\zeta^A)$, it is clear that
the largest $\delta$-block is clearly given by $\zeta^A$.

Let us prove Part 2.
First let us prove that 
$\delta^n$ is the \immpred\ of $\beta^n$ for each $n\ge 2$.
By Part 3, an \immpred\ of \cf\ greater than $\beta^1$ uniquely exists.
Let $\mu:=\sum_{m=1}^M \zeta^m $ be the $\delta$-block
decomposition of the \immpred\ of $\beta^n$.
Then, $\beta^{n-1} \le_{\text{a}} \zeta^M$, and hence, $\zeta^M$ 
has index $n-1$.
If $\zeta^M$ is the maximal $\delta$-block at index $n-1$, and hence, $M=1$, then $\zeta^M=\delta^n$, which proves the assertion.
If $\zeta^M$ is a proper $\delta$-block at index $n-1$,
then  
there is an index $i\le n-1$ such that $\zeta^M_i < \delta^n_i$ and 
$\zeta^M_k = \delta^n_k$ for $i<k\le n-1$.
However, this implies that $\mu \aord \delta^n$, which contradicts that $\mu$ is 
not the \immpred\ of $\beta^n$.
This concludes the proof of $\delta^n = \hbeta^n$.

Let $\mu:=\sum_{m=1}^M \zeta^m $ be the $\delta$-block decomposition of $\mu\in E$ with supports $[i_m,n_m]$
with $n_m+1 = i_{m+1}$ for $1\le m\le M-1$.
If $\zeta^1$ is a proper $\delta$-block, then clearly $ \beta^1 +\mu$ is the unique \immsucc\ of $\mu$ in $E$.
Suppose that $\zeta^1=\delta^{n_1+1}$, and
there is $\mu^1\in E$ such that
\begin{gather*}
\mu:=\zeta^1 + \sum_{m=2}^M \zeta^m \aord\quad \mu^1\quad \aord \quad
\mu^2:=\beta^{n_1+1} +\sum_{m=2}^M \zeta^m
=\beta^{i_2} +\sum_{m=2}^M \zeta^m.\\
\text{Then,}\quad
\mu_k =\mu^2_k\text{ for all }k> i_2
\implies \mu^1_k = \mu_k =\mu^2_k\text{ for all }k> i_2, \\
\text{and}\quad
\mu_{i_2}=\zeta^2_{i_2} \le\mu^1_{i_2} \le \mu^2_{i_2}= 1+\zeta^2_{i_2}.
\end{gather*}
If $\mu^1_{i_2}= \mu^2_{i_2}$, then
$\mu^1_k < \mu^2_{k}=0$ for some $k<i_2$, which is not possible.
So, it follows that $ \mu_{i_2}=\mu^1_{i_2}=\zeta^2_{i_2}$,
and hence, there is a \cf\ $\tau$ of order $n_1$ that is larger than
$\zeta^1=\delta^{n_1+1}$ such that
$\mu^1 = \tau + \sum_{m=2}^M \zeta^m$.
If $M\ge 2$, by Part 1, $\zeta^M$ is an $\delta$-block appearing in the $\delta$-block decomposition of $\mu^1$, and $\tau+ \sum_{m=2}^{M-1} \zeta^m$ remains
to be a member of $E$. By repeating this process, we find that $\tau$ is a member of $E$, which is true for $M=1$ as well.
However, $\delta^{n_1+1} \aord \tau \aord \beta^{n_1+1}$
contradicts that $\delta^{n_1+1}$ is the \immpred\ of
$\beta^{n_1+1}$.
\iffalse
Let $\theta$ be the largest \Eb\ with
support $[a,n_1]$ appearing in the \Eb\ decomposition of $\delta$.
If $\theta$ is a proper \Eb, $\theta_a<\hbeta^{n_1+1}_a$ and $\theta_k=\hbeta^{n_1+1}_k$
for $a<k\le n_1$, but this implies that
$\delta\aord \zeta^1=\hbeta^{n_1+1}$, i.e.,
$\mu^1 \aord \mu$, which contradicts $\mu \aord \mu^1$.
If $\theta$ is a maximal \Eb, then $\theta=\delta=\hbeta^{n_1+1}$, i.e.,
$\mu^1 =\mu$, which contradicts $\mu \aord \mu^1$.
\fi
Therefore, we conclude that
$\mu^2$ is the \immsucc\ of $\mu$ in $E$.

Let us prove that $E$ is \zec.  Part 3 of this lemma proves the property
described in Definition \ref{def:zec-N}, Part 1.
Part 2 of this lemma proves the property
described in Definition \ref{def:zec-N}, Part 2.

\end{proof}

\begin{prop}\label{prop:E-cE-rule}
Let $\cE$ be the collection defined in Lemma \ref{lem:decomposition} with
\immpred s $\hbeta^n$ for $n\ge 2$, 
and let $E$ be the collection defined in  Lemma \ref{lem:E-unique}
where   $\delta^n = \hbeta^n$ for each $n\ge 2$.
Then, the following are true:
\begin{enumerate}
\item If $\mu\in E\cap \cE$, then
its \immsucc\ in $E$ is equal to its \immsucc\ in $\cE$.
\item
The collection $E$ is a subset of $\cE$.
\end{enumerate}
\end{prop}
\begin{proof}
Since $0\in E \cap \cE$, the intersection is non-empty.
Let $\mu^0$ be a \cf\ in $E\cap \cE$, and let $\mu^1$ be
the \immsucc\ of $\mu^0$ in $E$.
Let us show that the \immsucc\ of $\mu^0 $ in $\cE$ is $\mu^1$.
Let $\mu^0=\sum_{m=1}^M \zeta^m$ be the \Eb\ decomposition.
Suppose that $\zeta^1$ is a proper \Eb.
Then, by Part 2 of Lemma \ref{lem:E-unique},  
we have 
$\mu^1 = \beta^1 + \mu^0$.
By contradiction, assume that the \immsucc\ 
$\mu^0$ in $\cE$ is not $\mu^1$.
 Then, by Part 2 of Definition \ref{def:zec-N}, 
there is $n\ge 2$ such that $\mu^0 \equiv \hbeta^n \reS [1,n)$
and the \immsucc\ of $\mu^0$ in $\cE$ is given by
$\beta^n + \resab{[n,\infty)}(\mu^0)\in \cE$, then, by Corollary \ref{cor:decomposition},
$\resab{[n,\infty)}(\mu^0)$ has the \Eb\ decomposition into proper \Eb s $\xi^m$ for $m=1,\dots,T$,
and hence, $\mu^0= \hbeta^n + \sum_{k=1}^T \xi^m$, which
is an \Eb\ decomposition.
However, this violates the uniqueness of \Eb\ decompositions since
the decomposition $\mu^0= \sum_{m=1}^M \zeta^m $ begins with a proper
\Eb, and hence, the \immsucc\ of $\mu^0$ in $\cE$ is given by
$\beta^1 + \mu^0$, which coincides with $\mu^1$ by Lemma \ref{lem:E-unique}, Part 2.

Suppose that $\zeta^1=\hbeta^{n_1+1}$.
Then, by Part 2 of Lemma \ref{lem:E-unique},  
we have 
$\mu^1 = \beta^{n_1+1}+ \sum_{m=2}^M \zeta^m$.
By contradiction, assume that the \immsucc\ of $\mu^0$ in $\cE$ is given by
$\beta^1 + \mu^0\in \cE$. 
Then, by Corollary \ref{cor:decomposition},
$\resab{[1,\infty)}(\mu^0)$ has the \Eb\ decomposition into proper \Eb s,
i.e., $\mu^0= \sum_{k=1}^T \xi^m$ where $\xi^m$
are proper.
However, since the smallest \Eb\ of $\mu^0 = \sum_{m=1}^M \zeta^m $ is maximal, it violates the uniqueness of \Eb\ decompositions of members of $E$, and hence, the \immsucc\ of $\mu^0$ in $\cE$ is given by
$\beta^t + \resab{[t,\infty)}(\mu^0)$ for some $t$ where
$\hbeta^t \equiv \mu^0 \reS [1,t)$.
By Corollary \ref{cor:decomposition},
$ \resab{[t,\infty)}(\mu^0)=\sum_{m=1}^\infty \theta^m$
is a decomposition into proper \Eb s, and hence,
$\mu^0=\hbeta^t + \sum_{m=1}^\infty \theta^m$ is an \Eb\ decomposition.
By the uniqueness of the \Eb\ decomposition, we have
$t=n_1+1$, and the \immsucc\ of $\mu^0$ in $\cE$
coincides with $\mu^1$ by Lemma \ref{lem:E-unique}, Part 2.
For both cases, $\mu^1$ is the \immsucc\ of $\mu^0$ in
$\cE$ as well as in $E$.

Let us show that $E \subset \cE$.
By Lemma \ref{lem:E-unique},
  $E$ is \zec, and by Lemma \ref{lem:seq}, $E=\set{\tau^n : n\ge 0}$ where $\tau^n$ are the \cf s defined in Lemma \ref{lem:seq}.
  By Part 2 above, if  $\tau^n \in E\cap \cE$, then $\tau^{n+1} \in 
  E\cap \cE$.  
  Since $\tau^0\in E\cap \cE$, by induction, we prove $E \cap\cE$.

\end{proof}
 
\begin{lemma}\label{lem:E-cE-rule}
Let $\cE$ and $E$ be the collections defined in Proposition \ref{prop:E-cE-rule}.
Then, $\cE$ is a subset of $E$.  In particular, $\cE = E$.
\end{lemma}

\begin{proof}
Let $\mu $ be a \cf\ in $\cE$.
Let us show that $\mu $ has an \Eb\ decomposition.
If $\td\mu = \beta^1+\mu$, then by Corollary \ref{cor:decomposition},
$\mu$ has a decomposition into proper \Eb s, and hence,
$\mu\in E$.
Suppose that
$\hbeta^T \equiv \mu \reS [1,T)$ and
$\td\mu=\beta^T+ \resab{[T,\infty)}(\mu)$.
Then, by Corollary \ref{cor:decomposition}, we conclude that
$\resab{[T,\infty)}(\mu)\in\cE$ has a block decomposition into
proper \Eb s $\zeta^m$ for $m\ge 1$.
Thus, $\mu = \hbeta^T + \sum_{m=1}^\infty \sum \zeta^m$ makes
an \Eb\ decomposition, and hence,
$\mu\in E$.
By Proposition \ref{prop:E-cE-rule}, we prove that 
$\cE = E$. 

\end{proof}

\subsection{Proofs of Theorem \ref{thm:block-def-N} and \ref{thm:block-R} } 

Let's prove Theorem \ref{thm:block-def-N}, Part 2. 
Let  $\cE$ be a \zec\ collection for positive integers, and
let $E$ be the \aorcoll\ collection determined by the \immpred s $\hbeta^n$ 
of $\cE$ for $n\ge 2$.
If $\zeta^m$ for $m=1,\dots,M$ are \Eb s defined in Theorem \ref{thm:block-def-N}, Part 2, then $\mu:=\sum_{m=1}^M \zeta^m$ is a member of $E$, and by 
Lemma \ref{lem:E-cE-rule}, $\mu$ is a member of $\cE$.

Let's prove Theorem \ref{thm:block-def-N}, Part 1. 
Let $\cE$ be a \zec\ collection for positive integers, and let $E$
be the \aorcoll\ collection determined by $\hbeta^n$ of $\cE$ for $n\ge 2$. 
Then by Lemma \ref{lem:E-cE-rule}, $\cE = E$, and this immediately proves the property of Theorem \ref{thm:block-def-N}, Part 1 (a).
The property of Part 1(b) follows from Lemma \ref{lem:E-unique}, Part 2 since $\cE = E$.
Let us prove the if-part of Theorem \ref{thm:block-def-N}, Part 1.
Let $\cE$ be an \aorcoll\ collection of \cf s satisfying properties described in
Theorem \ref{thm:block-def-N}, Part 1.  Then, 
Theorem \ref{thm:block-def-N}, Part 1(a) implies that $\cE$ is the subset of 
the \aorcoll\ collection $E$ determined by $\hbeta^n$ for $n\ge 2$.
On the other hand, by Lemma \ref{lem:E-unique}, $E$ is \zec, and by 
Lemma \ref{lem:seq},    $E=\set{\tau^n : n \ge 0}$ where $\tau^n$ are the \cf s defined in Lemma \ref{lem:seq}.
By  Lemma \ref{lem:E-unique}, Part 2, each $\tau^n$ is obtained by applying 
the rule given in Theorem \ref{thm:block-def-N}, Part 1(b), beginning with $\tau^0$, and hence,
$E$ is a subset of $\cE$.  Thus, $E=\cE$ is \zec.

Let us prove Corollary \ref{cor:generate-zec}.
Let $\cE$ be the collection defined in the corollary.
Then by Lemma \ref{lem:E-unique}, $\cE$ is \zec, and  $\hbeta^n = \delta^n$ for each $n\ge 2$.
If $\mathcal F$ is another \zec\ collection for positive integers with the \immpred s 
$\hbeta^n=\delta^n$ for $n\ge 2$, then by Lemma \ref{lem:E-cE-rule},   $\mathcal F=\cE$, and hence,
the \aorcoll\ collection $\cE$ is the only \zec\ collection 
for positive integers such that  $\hbeta^n=\delta^n$ for each $n\ge 2$.

Let us prove Theorem \ref{thm:block-R}.
Let $\cE$ be a \zec\ collection   for the interval $\OI$.
The proofs of Corollary \ref{cor:decomposition},
Lemma \ref{lem:E-unique},
Proposition \ref{prop:E-cE-rule},
and Lemma \ref{lem:E-cE-rule}
remain valid for
 \cf s in $\cE^M$ for $M\ge 1$ simply by
reversing the order of the values of the \cf s.
Let us make references to these versions by
{\it the \dver\ of Statement X}, e.g.,
the \dver\ of Corollary \ref{cor:decomposition}.

Suppose that $\cE$ is \zec\ for the interval $\OI$, and
let us prove Theorem \ref{thm:block-R}, Part 1 (b).
Let $\ep\in \cE$, and let $M_k$ for $k\ge 1$ be the finest sequence guaranteed by
Definition \ref{def:immsucc-R}, Part 4, i.e.,
$M_1$ is the smallest choice, and recursively,
$M_{k+1}$ is the smallest choice
that is greater than $M_k$.
By the \dver\ of Corollary \ref{cor:decomposition}, for each $\ep^{k}:=\ep \reS M_k$,
we find a decomposition of $\ep^{k}$ into proper \Ebb s.
Notice that if $\ep\ne 0$ and $\ep^k=\sum_{m=1}^T \zeta^m$ is the \Ebb\ decomposition, then by the \dver\ of
Lemma \ref{lem:E-unique}, Part 1,
the largest non-zero \Ebb\ of the \Ebb\ decomposition of $\ep^{k+1}$ for sufficiently large $k$ is equal to the largest non-zero \Ebb\ of $\ep^k$.
This implies that $\ep^{k+1}=\sum_{m=1}^{T+1} \zeta^m$ is the
\Ebb\ decomposition for some proper
\Ebb\ $\zeta^{T+1}$ while the smaller \Ebb s are the same ones of $\ep^k$.
This proves Theorem \ref{thm:block-R}, Part 1 (b).
Let 
$E$ be the \dorcoll\ collection determined by $\set{\bbeta^n : n\ge 1}$.
The \dver\ of Lemma \ref{lem:E-cE-rule} applied to $\cE^M$
and $E^M$ for each $M \ge 1$ implies Part 1 (c), and this concludes the proof of the only-if part of Theorem \ref{thm:block-R},
Part 1.

Let us prove the if-part of Theorem \ref{thm:block-R}.
Let $\cE$ be an \dorcoll\ collection of \cf s satisfying properties described in Theorem \ref{thm:block-R}, Part 1, and let 
$E$ be the \dorcoll\ collection determined by $\set{\bbeta^n : n\ge 1}$.
As argued for the collections for positive integers, we find
$\cE^M = E^M$ for each $M\ge 1$, and hence, $\cE = E$.
The \dver\ of Lemma \ref{lem:E-unique} implies
Definition \ref{def:immsucc-R}, Part 1,
and Definition \ref{def:immsucc-R}, Part 2 and 3
follow from the statements of Theorem \ref{thm:block-R}, Part 1 (a) and (c)
since $\cE = E$.
Let us prove Definition \ref{def:immsucc-R}, Part 4.
Let  
$\ep=\sum_{m=1}^\infty \zeta^m$ be
the proper \Eb\ decomposition as in Theorem \ref{thm:block-R}, Part 1 (b) with
support $[n_m,i_m]$.
Then, the increasing sequence $M_k:=i_k$ satisfies
the condition described in Definition \ref{def:immsucc-R}, Part 4.
Suppose that there are a \cf\ $\ep$ and an increasing sequence $M_k:=i_k$ such that
the condition described in Definition \ref{def:immsucc-R}, Part 4
is satisfied,
and let us prove that $\ep\in\cE$.
For each $M_k$, the \dver s of Corollary \ref{cor:decomposition} and
Lemma \ref{lem:E-unique} imply the existence and uniqueness of proper \Ebb s $\zeta^j$
for $j=1,2,\dots$ such that
$\ep \equiv \sum_{j=1}^\infty \zeta^j \reS M_k$ for $k=1,2,\dots$.
Hence, $\ep = \sum_{j=1}^\infty \zeta^j\in E=\cE $.

Let us prove Theorem \ref{thm:block-R}, Part 2, and let
$\cE$ be \zec.
If $\ep = \sum_{j=1}^\infty \zeta^j$ is the sum of disjoint proper \Ebb s
with support $[n_m,i_m]$, then the increasing sequence $M_k:=i_k$
for $k\ge 1$ satisfies the condition of Definition \ref{def:immsucc-R}, Part 4, and hence, $\ep\in\cE$.

Let us prove Corollary \ref{cor:generate-zec-real}. 
Let $\cE$ be the collection defined in the corollary.
By the definition of the \dorcoll\ collection determined by $\set{\delta^n : n\ge 2}$,
the property described in Theorem \ref{thm:block-R}, Part 1(a,b) is trivially satisfied.
By the descending version of Lemma \ref{lem:E-unique}, Part 2,
the properties described in Theorem \ref{thm:block-R}, Part 1(c) are   satisfied
for $\cE^M$ for all $M\ge 1$.
Thus, by Theorem \ref{thm:block-R}, $\cE$ is \zec.
Let $M$ be an arbitrary positive integer.
By the descending version of  Lemma \ref{lem:E-unique}, Part 2, the \immsucc\ of $\delta^n \reS M$ in $\cE^M$ is 
$\beta^{n-1}$ for each $n\ge 2$, and $\delta^1 \reS M$ in $\cE^M$ is the largest \cf\ in $\cE^M$.
Hence, $\hbeta^n = \delta^n$ for each $n\ge 1$.
Let $\mathcal F$ be another \zec\ collection for $\OI$ with the same 
$\bbeta^n$ for $n\ge 1$.
Then, by the descending version of Lemma \ref{lem:E-cE-rule},
$\mathcal F^M = \cE^M$ for each $M\ge 1$.  Hence, $\mathcal F = \cE$.

\subsection{The integer case}\label{sec:integer}
 
Let us prove Theorem \ref{thm:integers}.
Part 1 (a) follows immediately from Algorithm \ref{alg:E}, and we prove
Part 1 (b) here.
Recall that given $\delta\in\cE$, we denote by $\td\delta$ the immediate successor of $\delta$ in $\cE$.
\begin{lemma}\label{lem:add-one}
Let $\cE$ and $Q$ be the collection and the sequence defined in Theorem \ref{thm:integers}, Part 1 (b), respectively.
If $\delta\in\cE$, then $Q_1 +\sum \delta Q = \sum \td\delta Q$.
\end{lemma}

\begin{proof}
If
$\td\delta=\beta^1 + \delta$, and the statement
$Q_1 + \sum\delta Q = \sum \td\delta Q$ is clearly true.
If $\td\delta\ne\beta^1 + \delta$, then
$\delta =\hbeta^n + \resab{[n,\infty)}(\delta)$,
and $\td\delta=\beta^n + \resab{[n,\infty)}(\delta) $.
Thus,
\begin{align*}
Q_1 + \sum \delta Q
&= Q_1 + \sum \hbeta^n Q + \sum_{k=n}^\infty \delta_k Q_k
=Q_n + \sum_{k=n}^\infty \delta_k Q_k=\sum \td\delta Q
\end{align*}
where the second equality is obtained
by using the recursive definition of $Q$.
\end{proof}

Let us prove that the function $f : \cE-\set{0} \to \nat$
given by $f(\delta) = \sum \delta Q$ is bijective.
Let $\delta^1:=\beta^1$, and recursively define $\delta^{k+1}$
to be the \immsucc\ of $\delta^{k}$ for $k\ge 1$, so that 
$\cE = \set{0}\cup\set{\delta^k : k\ge 1}$ by Lemma \ref{lem:seq}.
Let us use the induction to prove $f(\delta^n) = n$ for all $n\ge 1$.
First, $f(\delta^1)= \sum \delta^1 Q = \sum \beta^1 Q = 1$, and
suppose that there is an index $n\ge 1$ such that
$f(\delta^k) = k$ for all $k\le n$.
Then,
$f(\delta^{n+1}) = \sum \delta^{n+1} Q = Q_1 + \sum \delta^n Q$
by Lemma \ref{lem:add-one}, and hence, by the induction hypothesis,
$f(\delta^{n+1}) = Q_1 + n = n+1$.
Thus, each integer $n\in\nat$ is equal to $\sum \delta Q$ for a unique
non-zero \cf\ $\delta\in \cE$.
This concludes the proof of Part 1 (b).

Part 2 (a) follows immediately from Part 1.
Let us prove Part 2 (b). Note that for the first sentence of Part 2 (b),
the collection is not necessarily \zec.
Let $Q$ and $Z$ be increasing fundamental sequences in $\nat$ that generate
a subset $Y$ under the $\cE$-condition, i.e., $
X_Q^\cE = X_Z^\cE$.
Let us use the induction to show that $Q_n=Z_n$ for each $n\ge 1$.
Since $Q_1$ and $Z_1$ are the smallest integers in $Y$, we have
$Q_1=Z_1$.
Suppose that
$Q_k=Z_k$ for all $1\le k< n$ where $n\ge 2$.
Suppose that $Z_n < Q_n$.
Then, there is a \cf\ $\mu\in \cE$ such that $Z_n = \sum \mu Q$.
If $\ord(\mu)\ge n$, then $Z_n=\sum \mu Q \ge Q_n$,
which contradicts $Z_n < Q_n$.
Thus, $\ord(\mu)<n$.
By the induction hypothesis, $Z_n=\sum \mu Q=\sum_{k=1}^{n-1} \mu_k Q_k =\sum \mu Z$.
Since $Z_n$ and $\sum \mu Z$ are distinct $\cE$-expansions,
it contradicts that $Z$ has the \Eunique.
We conclude that $Q_n \le Z_n$.
If $Q_n< Z_n$, a similar contradiction is derived, and we conclude that
$Q_n = Z_n$.
The second sentence of Part 2 (b) follows immediately from the choice
of $Q$ we made for Part 2 (a), and the uniqueness of increasing sequences
we just proved above.

\subsection{The real number case}\label{sec:real} 
We prove Theorem \ref{thm:real} in this section. 
Let $\cE$ be a \zec\ collection     for the unit interval $\OI$, and recall from Section \ref{sec:D-R} the maximal \cf\ $\bbeta^n$ of order $n$.
For convenience of stating results in this section, we define $Q_0=1$.

\begin{lemma}\label{lem:inequality-1}
Let $Q$ be a decreasing fundamental sequence for the $\cE$-interval $\OI$.
If $n\ge 1$, then
$\sum\bar\beta^n Q \le Q_{n-1}$.

\end{lemma}
\begin{proof}
For $n=1$,
by definition of being a fundamental sequence for the $\cE$-interval $\OI$,
we have $ \sum_{k=1}^M\bbeta^1_k Q_k<1 $ for large index $M$,
and hence, $\sum\bbeta^1 Q=\lim_{M\to\infty} \sum_{k=1}^M\bbeta^1_k Q_k\le 1 $.
For some $n\ge 2$, and
suppose that $\sum\bbeta^n Q > Q_{n-1}$.
Then, there is a proper \Ebb\ $\xi$ at index $n$ such that $Q_{n-1} < \sum \xi Q$.
Since there are only finitely many proper \Ebb s that are less than $\xi$,
find the largest proper \Ebb\ %
$\zeta$ with index $[n,M]$
such that
\begin{equation} \label{eq:cross}
\sum\zeta Q \le Q_{n-1} < \sum(\zeta+\beta^M) Q.
\end{equation}
Then, by the uniqueness of $\cE$-expansions with the fundamental sequence,
we have a strict inequality
$\al:=\sum \zeta Q<Q_{n-1}<\sum\zeta Q+ Q_M$.
Thus, $0< Q_{n-1}-\al < Q_M$, and by the existence of
$\cE$-expansions of numbers in $\OI$,
we have
$0<Q_{n-1}-\al = \sum \delta Q<Q_M$ for $\delta\in\cE$,
and $\ell:=\ord(\delta)>M$; if $\ell \le M$, then $\sum \delta Q
\ge \delta_{\ell} Q_\ell \ge Q_M$.
However, $Q_{n-1}=\al +\sum \delta Q
=\sum\zeta Q +\sum \delta Q$, and
by Theorem \ref{thm:block-R},
the last expansion is an $\cE$-expansion.
Since $Q_{n-1}$ is an $\cE$-expansion as well, this contradicts the uniqueness of $\cE$-expansions, and
we prove that $\sum\bbeta^n Q \le Q_{n-1}$.
\end{proof}
Let us introduce a lemma as a preparation for the proof of Proposition
\ref{prop:max-sum}.
\begin{lemma}\label{lem:telescoping-sum}
Let $Q$ be a decreasing fundamental sequence for the $\cE$-interval $\OI$.
Let $\sum_{k=1}^\infty \zeta^k$ be an \Ebb\ decomposition with
support $[m_k,a_k]$ such that $m_k - 1=a_{k-1}$ for all $k\ge 2$.
Then, $\sum_{k=1}^\infty \sum\zeta^k Q< Q_{m_1-1}$.
\end{lemma}
\begin{proof}
By the definition of
a proper \Ebb, $0\le \zeta^k_{a_k}<\bbeta^{m_k}_{a_k}$,
which implies that $\bbeta^{m_k}_{a_k}>0$.
Then, for each $k=1,2,\dots$
\begin{gather}
\sum\zeta^k Q
	 =\sum_{j=m_k}^{a_k} \bbeta^{m_k}_j Q_j - b_k Q_{a_k},\quad
	\text{for some integer $ b_k\ge 1$.}\notag\\
\intertext{By Lemma \ref{lem:inequality-1},
$\sum \bbeta^{m_k} Q \le Q_{m_k-1}$, and it follows that }
\sum\zeta^k Q
\le \sum \bbeta^{m_k} Q - Q_{a_k}
- \sum_{j=a_k+1}^\infty \bbeta^{m_k}_j Q_j
< Q_{m_k-1} - Q_{a_k} \notag
\end{gather}
where $ \sum_{j=a_k+1}^\infty \bbeta^{m_k}_j Q>0$ since $\bbeta^{m_k}$ has infinite support.
Notice that
$ Q_{m_k-1} - Q_{a_k}=Q_{a_{k-1}} - Q_{a_k} $ for $k\ge 2$, and that the above inequalities hold for $\zeta^k=0$ as well.
\begin{align}
\implies \sum_{k=1}^\infty \sum\zeta^k Q
&=\sum\zeta^1 Q + \sum_{k=2}^\infty \sum\zeta^k Q\notag\\
& < Q_{m_1-1} - Q_{a_1}
+\sum_{k=2}^\infty (Q_{a_{k-1}} - Q_{a_k})\notag\\
\intertext{Simplifying the telescoping sum, we have}
& = Q_{m_1-1} - Q_{a_1}
+Q_{a_1}=Q_{m_1-1} .
\label{eq:upper-bound}
\end{align}
\end{proof}
\begin{prop}\label{prop:max-sum} 
Let $Q$ be a decreasing fundamental sequence for the $\cE$-interval $\OI$.
Then 
$\sum\bbeta^m Q = Q_{m-1}$ for each index $m\ge 1$, and  $\sum\delta Q <\sum\bbeta^m Q$  for all
$\delta\in \cE$ of order $ m$. 
\end{prop}

\begin{proof}
By Lemma \ref{lem:inequality-1}, $Q_{m-1} \ge \sum\bbeta^m Q$.
Suppose that $Q_{m-1} > \sum\bbeta^m Q$ for some $m\ge 1$,
and let $ \sum\bbeta^m Q = \sum_{k=1}^\infty \sum\zeta^k Q$
be the \Ebb\ decomposition into proper \Ebb s $\zeta^k$
with support $[m_k,a_k]$ such that
$\zeta^1$ is a non-zero \Ebb\ and $m_k-1 = a_{k-1}$ for $k\ge 2$.
Then, $Q_{m-1} > \sum\bbeta^m Q> Q_{m_1}$ implies $m_1>m-1$, i.e., $m_1\ge m$.

On the other hand, since $\bbeta^m$ has infinite support,
$Q_m<\sum\bbeta^m Q= \sum_{k=1}^\infty \sum\zeta^k Q< Q_{m_1-1}
$ by Lemma \ref{lem:telescoping-sum}, and it implies that
$m_1-1<m$, i.e., $m_1\le m$.
Since we found $m_1\ge m$ earlier, we have $m=m_1$.
Thus, $\zeta^1_s = \bbeta^m_s$ for $m\le s< a_1$ , so
\begin{gather*}
\sum \bbeta^m Q = \sum\zeta^1 Q + \sum_{k=2}^\infty \sum\zeta^k Q
\implies
Q_{a_1} +\sum_{k=a_1+1}^\infty \bbeta^m Q
\le \sum_{k=2}^\infty \sum\zeta^k Q.\\
\intertext{By Lemma \ref{lem:telescoping-sum}, we find
$\sum_{k=2}^\infty \sum\zeta^k Q < Q_{m_2-1} $,
and hence,}
\implies
Q_{a_1} <
Q_{a_1} +\sum_{k=a_1+1}^\infty \bbeta^m Q
\le \sum_{k=2}^\infty \sum\zeta^k Q < Q_{m_2-1}
= Q_{a_1} .
\end{gather*}
Since this is a contradiction,
we conclude $\sum\bbeta^m Q \ge Q_{m-1}$, and
by Lemma \ref{lem:inequality-1}, we prove that
it is an equality.

Let us prove the second statement.
Let $\sum \delta Q = \sum_{k=1}^\infty \sum\zeta^k Q$ be the
\Ebb\ decomposition where $\ord(\delta)=\ord(\zeta^1)=:m$.
Then, by Lemma \ref{lem:telescoping-sum},
\begin{gather*}
\sum \delta Q= \sum_{k=1}^\infty \sum\zeta^k Q
< Q_{m-1}=\sum\bbeta^m Q.
\end{gather*}

\end{proof}

\subsubsection{Proof of Theorem \ref{thm:real}, Part 1}
Let us prove  the only-if part of Theorem \ref{thm:real}, Part 1 (b).
Assume the notation and context of Theorem \ref{thm:real}, Part 1 (b),
and let $Q$ be a decreasing fundamental sequence for $\OI$.
Then, by Proposition \ref{prop:max-sum}, the sequence satisfies the 
equality (\ref{eq:full-recursion-R}).

Let us prove the if-part of Part 1 (b).
For the existence of an $\cE$-expansion, we shall use the greedy
algorithm in terms of proper \Ebb s.
Let $x\in \OI$.
Then, there is an index $n\ge 1$
such that $Q_n \le x < Q_{n-1}$.
By the equality (\ref{eq:full-recursion-R}), $\sum \bbeta^n Q = Q_{n-1}$, and
hence, there is a proper \Ebb\ $\xi$ of order $n$ such that
$x< \sum\xi Q < Q_{n-1}$.
Thus, by the finiteness of the number of proper \Ebb s at index $n$ that are less than $\xi$,
there must be
a largest proper \Ebb\ $\zeta $ at index $n$ such that $
\sum \zeta Q \le x$.
Note here that $\zeta$ is in fact the largest proper \Ebb\ in $\cE$ such that
$\sum \zeta Q \le x$.
Let $\zeta^1:=\zeta$, $x_1:=x-\sum \zeta^1 Q$, and $x_0:=x$,
and recursively define $x_m:= x_{m-1} - \sum\zeta^m Q$
for $m\ge 1$
where $\zeta^m$ is the largest proper \Ebb\ such that $
\sum \zeta^m Q \le x_{m-1}$.
Then,
$x =\sum_{m=1}^M\sum \zeta^m Q + x_M$ for some $x_M\in[0,1)$ for all $M\ge 1$.
By the choice of $\zeta^m$, we have
$\zeta^m\ne 0$ if $x_{m-1}>0$, and if $x_{M-1}=0$ for some $M\ge 1$,
then $\zeta^m=0$ for all $m\ge M$.

Given $m\ge 1$, suppose that $\zeta^m$ and $\zeta^{m+1}$ are
non-zero \Ebb s, and have supports $[n,j]$ and $[n',j']$, respectively.
Let us prove that $j<n'$.
Suppose that $n'\le j$, and note that $\zeta^{m+1}_{n'}\ge 1$.
Then,
$x_{m+1}=x_m - \sum\zeta^{m+1} Q
\le x_m - Q_{n'}
\le x_m - Q_j
= x_{m-1}-\sum\zeta^m Q - Q_j$.
Notice that $\sum\zeta^m Q + Q_j
=\sum_{k=n}^j \zeta^m_k Q_k + Q_j $.
If $\zeta^m_j+\beta^j<\bbeta^n_j$, then
$\zeta^m + \beta^j$ forms a proper \Ebb\ with the same
support $[n,j]$.
If $\zeta^m_j+\beta^j=\bbeta^n_j$, since $\bbeta^n$ has infinite support, and there is a smallest index $\ell >j$
such that $\bbeta^n_\ell>0$.
Then,
$\zeta':=\zeta^m + \beta^j$ also forms a proper \Ebb\ with the support $[n,\ell]$ where $\zeta'_{\ell}=0
<\bbeta^n_\ell$.
Hence, $0\le x_{m+1}\le x_{m-1} -\sum\zeta' Q$ where $\zeta^{m} \dord
\zeta'$. This contradicts the choice of $\zeta^m$.
Thus, $\zeta^k$ for $k\ge 1$ are decreasing proper \Ebb s with disjoint supports.

Recall $x =\sum_{m=1}^M\sum \zeta^m Q + x_M$ for all $M\ge 1$,
and let us prove that $x_M\to 0$ as $M\to \infty$.
If $s$ is the largest index of the support of $\zeta^M$, then
$x_M<Q_s$; otherwise, we would have chosen a proper \Ebb\ larger than or equal to
$\zeta^M + \beta^s$.
Since $Q_s\to 0$ as $M\to \infty$, we prove that $x_M\to 0$ as well,
and hence,
$x = \sum_{m=1}^\infty\sum \zeta^m Q $.
Thus, we proved that $x$ has an $\cE$-expansion.
The uniqueness of such an expansion follows immediately from the property that if $y$ is a real number in $\OI$ and $y=\sum_{m=1}^\infty\sum\xi^m Q$ is an \Ebb\ decomposition such that $\xi^1\ne 0$,
then $\xi^1$ is the largest proper \Ebb\ such that $\sum \xi^1 Q \le y$.
Let us prove this property.
Let $\xi^m$ be the proper \Ebb s described above.
Write
$
y =\sum_{m=1}^\infty\sum\xi^m Q
	=\sum \xi^1 Q + \sum_{m=2}^\infty\sum\xi^m Q$,
and let $[N,K]$ be the support of $\xi^1$.
	If $ \sum_{m=2}^\infty\sum\xi^m Q$ is zero, then we are done, and if not,
	then by Proposition \ref{prop:max-sum}, this infinite sum is $<Q_{k_0}$ for some $k_0\ge K$.
It follows that
	$y<\sum \xi^1 Q + Q_{K }$.
	Since there is no proper \Ebb\ $\zeta$ such that
	$\xi^1 \dord \zeta \dord \xi^1 + \beta^K$, the inequalities
	$\sum\xi^1 Q \le y < \sum \xi^1 Q + Q_{K }$ implies that $\xi^1$
	is the largest proper \Ebb\ such that $\sum \xi^1 Q \le y$.
Therefore, if $x=\sum_{m=1}^\infty\sum\xi^m Q$ is the non-zero
\Ebb\ decomposition,
then each $\xi^m$ for $m\ge 1$ is the largest proper \Ebb\ that goes into $x_{m-1}$, and hence, $\xi^m = \zeta^m$ for all $m\ge 1$.
This concludes the proof of Theorem \ref{thm:real}, Part 1 (b).

For Part 1 (a), we use the greedy algorithm as explained
around Example \ref{exm:R-greedy}, and
we leave it to the reader.

\subsubsection{Proof of Theorem \ref{thm:real}, Part 2}
\label{sec:proof-theorem-real-part-2}

\begin{lemma}\label{thm:recurrence}
Let $Q$ be a decreasing sequence in $\OI$ converging to $0$.
Let $\cL$ be the \zec\ collection determined by a
\zml\ $L=(e_1,\dots,e_N)$.
The sequence $Q$ is a decreasing fundamental sequence for
the $\cL$-interval $\OI$
if and only if $Q$ is a sequence defined by the linear recurrence
$Q_n =\sum_{k=1}^N e_k Q_{n+k}= e_1Q_{n+1} + \cdots + e_N Q_{n+N}$ for all $n\ge 0$.
\end{lemma}
\begin{proof}
Suppose that $Q$ is a decreasing fundamental sequence for $\OI$, and
recall $Q_0=1$.
Given $n\ge 1$,
let $\bbeta^{n+1}$ be the maximal \cf\ of order $n+1$ for the \Lz\ condition.
Then, by Theorem \ref{thm:real}, Part 1 (b),
$
Q_n=\sum\bbeta^{n+1}Q =
\sum_{k=1}^{N-1} e_k Q_{n+k} + \hat e_N Q_{n+N}
+\sum_{k=n+N+1}^\infty \bbeta^{n+1}_k Q_k$.
Since the last term of the RHS
is the maximal \cf\ of order $n+N+1$, by Proposition \ref{prop:max-sum},
it follows
that
\begin{equation}\label{eq:short-recurrence}
Q_n =
\sum_{k=1}^{N-1} e_k Q_{n+k} + \hat e_N Q_{n+N}
+ Q_{n+N},
\end{equation} which is the short linear recurrence.

Suppose that $Q$ satisfies the recurrence (\ref{eq:short-recurrence}).
By repeatedly applying the recurrence to the terms $Q_{n+kN}$ for $k\ge 1$, we find $Q_n = \sum \bbeta^{n+1} Q$ for $n\ge 0$, and $\sum \bbeta^{n+1} > Q_{n+1}$ implies that it is
a decreasing sequence. By Theorem \ref{thm:real}, Part 1 (b),
it is a fundamental sequence for $\OI$.

\end{proof}

The following theorem is available in \cite[Theorem 1]{daykin3}.
\begin{theorem} [Daykin]\label{thm:daykin-2}
Let $Q$ is a decreasing sequence of positive real numbers given by the $L$-linear recurrence given in Lemma \ref{thm:recurrence}.
If the polynomial $f(z)=e_N z^N +\cdots + e_1 z -1$ has a complex root $\ome$ whose modulus is smaller than other complex roots, then
$\ome$ is a real number in $\OI$, and there is a positive real number $a$ such that
$Q_k=a \ome^k$ for all $k\ge 1$.
\end{theorem}
We shall show below that the hypothesis of Theorem \ref{thm:daykin-2} is always satisfied for \zml s, and we begin with the following lemma.

\begin{lemma}\label{lem:circle}
Let $f(z)=a_n z^n +\cdots + a_1 z+a_0$ be a polynomial in $\real[z]$, and suppose that
$a_k\ge 0$ for all $1\le k\le n$ such that there are indices $m<\ell$ with $\gcd(m,\ell)=1$ and $a_m a_\ell \ne 0 $.
If $f(\ome)=0$ for some positive real number $\ome$,
then there is no other complex zero of $f$ on the circle centered at the origin with radius $\ome$.
\end{lemma}

\begin{proof}
Suppose that there is a complex root $\al=\ome e^{i\theta}$ where $0<\theta<2\pi$, and
it implies that its complex conjugate
$\bar\al=\ome e^{-i\theta}$ is also a root of the polynomial.
\begin{gather*}
0=f(\ome e^{i\theta})
=a_0 + \sum_{k=1}^n a_k \ome^k e^{ik\theta},\quad
0=f(\ome e^{-i\theta})
=a_0 + \sum_{k=1}^n a_k \ome^k e^{-ik\theta}\\
\intertext{By adding the RHS to each other, we have}
0=2a_0 + \sum_{k=1}^n a_k \ome^k( e^{ik\theta}+e^{-ik\theta})
=2a_0 + \sum_{k=1}^n a_k \ome^k\cdot 2\cos(k\theta)
\implies 0= a_0 + \sum_{k=1}^n a_k \ome^k \cos(k\theta)\\
\intertext{On the other hand, $0=f(\ome)
\implies 0= a_0 + \sum_{k=1}^n a_k \ome^k$.
By subtracting the above equation from this one, we have}
0= a_0 + \sum_{k=1}^n a_k \ome^k-
\left(
a_0 + \sum_{k=1}^n a_k \ome^k \cos(k\theta)
\right)=\sum_{k=1}^n a_k \ome^k(1- \cos(k\theta)).
\end{gather*}
Since both $a_m$ and $a_\ell$ are positive and $1-\cos(k\theta)\ge 0$ for all $1\le k\le n$,
it follows $ \cos(m\theta)=\cos(\ell \theta)=1$.
This implies that
$\theta = 2m'\pi/m$ and $\theta =2\ell'\pi/\ell$ for two positive integers
$m'$ and $\ell'$ with $m'<m$ and $\ell'<\ell$.
However, it implies that $m'\ell = m\ell'$, and this equality on the natural numbers cannot happen if $\gcd(\ell,m)=1$
and $m'<m$ and $\ell'<\ell$.
\end{proof}

For example, $3z^n + z - 1$ has a unique positive real zero $\ome$ by Descartes\rq\ Rule of Sign,
and by Lemma \ref{lem:circle}, it has no other complex roots on the circle passing through $\ome$ centered at the origin
since $\gcd(a_n,a_1)=1$.

Let us prove Theorem \ref{thm:real}, Part 2.
For the characteristic polynomial $e_N z^N +\cdots e_1 z - 1$
for the linear recurrence,
we have a unique positive real zero $\ome$, and
Rouch\'e's Theorem on the setting
$\abs{e_N z^N +\cdots e_1 z }<1$ on a circle with radius smaller than $\ome$ implies that
$\ome$ has the smallest modulus.
Moreover, by Lemma \ref{lem:circle}, it is the only one with the smallest modulus.
By   Theorem \ref{thm:daykin-2},
we find that if $Q$ is a decreasing fundamental sequence
for the $\cL$-interval, then there is a positive real number $a$ such that
$Q_k=a\ome^k$ for all $k\ge 1$.
By Lemma \ref{thm:recurrence},
$1=Q_0=\sum_{k=1}^N e_k Q_k = a \sum_{k=1}^N e_k \ome^k =a$,
and hence, $a=1$.
It remains to show that the sequence $Q$ given by
$Q_k=\ome^k$ for all $k\ge 1$ is a fundamental sequence.
Since the sequence clearly satisfies the recurrence (\ref{eq:short-recurrence})
given in Lemma \ref{thm:recurrence},
it is a decreasing fundamental sequence.

\subsubsection{Proof of Proposition \ref{prop:dominant}}

The application of Rouch\'e's Theorem demonstrated in Section 
\ref{sec:proof-theorem-real-part-2} relies on
the constant term $e_0$ being positive, not particularly on $e_0=1$, so
we can say, $e_N z^N +\cdots e_1 z-e_0$ has no zeros inside the circle of radius $\ome$ if $\ome$ is the only positive zero of this polynomial that has coefficients in $\real$.
If we use Lemma \ref{lem:circle} with its conditions imposed on
$e_N,\dots,e_1$, we prove that $\ome$ is a unique zero with smallest modulus.
Since the polynomial in Proposition \ref{prop:dominant}
is the reciprocal version of $e_Nz^N +\cdots + e_1z - e_0$,
it proves that it is a dominant polynomial.

\subsection{The $p$-adic number case}\label{sec:p-adic}

When non-zero \cf s $\mu$ are used for $p$-adic integers,
the smallest index $n$ such that $\mu_n\ne 0$ is called
{\it the $p$-adic order of $\mu$}, denoted by $\ord_p(\mu)$. Let $\ord_p(0)=\infty$.
This definition is well-suited for the order of the $p$-adic numbers,
$\ord_p(x)$ for $x\in \zz_p$.

\subsubsection{Proof of Theorem \ref{thm:p-adic}, Part 1} 
Let $Q$ be a decreasing sequence such that $\sum \ep Q = \sum \delta Q$ for non-zero \cf s $\ep$ and
$\delta$ in $\cE$, and
let us claim $n:=\ord_p(\ep)=\ord_p(\delta)$.
If $n=\ord_p(\ep)<n':=\ord_p(\delta)$, then
$\ord_p(\sum \ep Q )=\ord_p(Q_n)<\ord_p(Q_{n'})=\ord_p(\sum \delta Q )$ since $\ep_n$ and $\delta_{n'}$ are $<p$ and $Q$ is decreasing.
This contradicts that $\ord_p(\sum \ep Q)=\ord_p(\sum \delta Q)$.
Without loss of generality, we conclude that $n=\ord_p(\ep)=\ord_p(\delta)$.
Suppose that there is a smallest positive index $s\ge n$ such that $\ep_s < \delta_s$.
Then, $\ep_s Q_s \equiv \delta_s Q_s \mod p^{m+1}$ where
$m=\ord_p(Q_s)\ge 0$, and hence,
$\ep_s (Q_s/p^m) \equiv \delta_s (Q_s/p^m) \mod p $.
Since $ Q_s/p^m\not\equiv 0\mod p$, we have
$\ep_s\equiv \delta_s \mod p$.
Since the values are $<p$, we find that $\ep_s=\delta_s$, which is a contradiction. Thus, without loss of generality, we prove that
$\ep=\delta$.

\subsubsection{Proof of Theorem \ref{thm:p-adic}, Part 2}

\begin{lemma}\label{lem:ord-Q-Z}
Let $p$ be prime, and let $\cE$ be an arbitrary collection of \cf s $\ep$ such that $\ep_k <p$ for all $k\ge 1$.
Let $Q$ and $Z$ be decreasing sequences in $\zz_p$
such that $X_Q = X_Z$. Then,
\begin{equation}
\ord_p(Q_k)=\ord_p(Z_k)\quad\text{ for all $k\ge 1$}.
\label{eq:QZ-order}
\end{equation}
\end{lemma}

\begin{proof}
Suppose that there is a smallest index $k\ge 1$ such that
$\ord_p(Q_k)<\ord_p(Z_k)$.
Write $Q_k = \sum \ep Z$ for a non-zero \cf\ $\ep \in \cE$, and
$n:=\ord_p(\ep)$.
Then, $\ep_j<p$ for all $j\ge 1$ implies that
$\ord_p(Q_k) =\ord_p(\sum\ep Z)= \ord_p(Z_n)$.
If $n= k$, then
$n=k$ contradicts the choice of $k$,
and if $n<k$, then $\ord_p(Q_k) = \ord_p(Z_n)=\ord_p(Q_n)$
by the choice of $k$, which contradicts the decreasing property of $Q$.
Thus, $n>k$, and it implies the following contradiction:
$\ord_p(Q_k)= \ord_p(Z_n)> \ord_p(Z_k)>\ord_p(Q_k)$.
Using a similar argument, we also derive a contradiction from
$\ord_p(Z_k)<\ord_p(Q_k)$.
Therefore, we prove (\ref{eq:QZ-order}).
\end{proof}

\begin{lemma}\label{lem:p-adic-decomposition}
Let $\cE$ be a \zec\ collection for $p$-adic integers.
If $\ep \in \cE$,
then $ \resab{[n,\infty)}(\ep) \in \cE$ for all indices $n\ge 1$.
\end{lemma}

\begin{proof}
Let $\cE_0$ be a \zec\ collection for positive integers whose completion is $\cE$.
Let $\ep\in \cE$.
Then, by definition, there is an increasing sequence of $M_k$ for $k\ge 1$
such that $\resab{[1,M_k]}(\ep)\in \cE_0$.
Notice that if $\zeta$ is an $\hbeta$-block at index $\ell$, then
$\resab{[n,\infty)}(\zeta)$ is an $\hbeta$-block at index $\ell$ for all indices $n\ge 1$.
Let $\resab{[1,M_k]}(\ep)=\sum_{t=1}^T \zeta^t$ be the unique
$\hbeta$-block decomposition as described in Theorem \ref{thm:block-def-N}.  
For arbitrary positive integer $n\le M_k$, there is an $\hbeta$-block $\zeta^s$
with $1\le s\le T$ such that the support of $\zeta^s$ is $[a,b]$ and 
$a\le n\le b$.  
Thus,
$\resab{[n,M_k]}(\ep)=\resab{[n,b]}(\zeta^s)+\sum_{t=s+1}^T \zeta^t$,
and this implies that given $n\ge 1$, we have $\mu^k:=\resab{[n,M_k]}(\ep)\in \cE_0$ for $k\ge 1$ and $M_k\ge \ord(\mu^k)$.
Since
$\resab{[n,\infty)}(\ep) \equiv \mu^k \reS{M_k}$, the \cf\ is
the limit of a sequence of \cf s in $\cE_0$, and hence, $\resab{[n,\infty)}(\ep)\in \cE$.

\end{proof}

\begin{lemma}\label{lem:QnZn}
Let $\cE$ be the \zec\ collection described in Theorem \ref{thm:p-adic}, Part 2.
Let $Q$ and $Z$ be decreasing sequences such that
$X_Q=X_Z$. Then, for each $n=1,2,\dots$,
there is an \Eb\ $\zeta$ such that $Z_n = \sum \zeta Q$,
$\ord_p(\zeta)=n$, and $\zeta_n=1$.
\end{lemma}
\begin{proof}
Let $\cE_0$ be a \zec\ collection for positive integers such that $\cE$ is the completion of $\cE_0$, and let $n$ be a positive integer.
Since $Z_n \in X_Q$, by the definition of the completion,
there is a non-zero \Eb\ $\zeta^1$ at index $\ell $ and a \cf\ $\ep$ such that
$Z_n = \sum \zeta^1 Q + \sum\ep Q$ and $\ep_s=0$ for $1\le s\le \ell$.
Since $m:=\ord_p(Z_n)=\ord_p(Q_n)$,
$\ord_p(\zeta^1)=n$,
and by Lemma \ref{lem:p-adic-decomposition}, $\ep \in \cE$.
Recall that \cf s in $\cE_0$ are ascendingly ordered, and
$\ord_p(\ep)\ge \ell+1$.
Since $\zeta^1 \in \cE$, it follows $\sum \zeta^1 Q = Z_n - \sum \ep Q
\in X_Q=X_Z$. It is clear
that $Z_n - \sum \ep Q=\sum\sig Z$ for
some $\sig\in\cE$ with $\ord_p(\sig)=n$ since $\ord_p( \sum\ep Q) \ge \ord_p(Q_{n+1})$ where $\ord_p(\ep)=\ord_p(\sum\ep Q)=\infty$ if $\ep=0$.
Then, $ Z_n - \sum \ep Q = \sum \sig Z$
implies $Z_n \equiv \sig_n Z_n \mod p^{m+1}$, i.e.,
$1\equiv \sig_n \mod p^{ 1}$.
Since $\sig_n<p$, we find $\sig_n=1$.
Write $\sig^0:=\resab{(n,\infty)}(\sig)$, and write
$ \sum \ep Q=\sum\delta Z$ for some $\delta \in \cE$,
so that $\ord_p(\ep)=\ord_p(\delta)\ge \ell+1$.
Then, $ Z_n + \sum\sig^0 Z = Z_n - \sum \ep Q = Z_n - \sum \delta Z$,
and hence,
$0 = \sum\sig^0 Z + \sum\delta Z
=\sum_{k=n+1}^\ell \sig^0_k Z_k
+\sum_{k=\ell+1}^\infty (\delta_k + \sig^0_k) Z_k$.
Notice that
$0\le \delta_k + \sig^0_k\le2\cdot ( p-1)/2=p-1$ for all $k\ge \ell+1$
and $0\le \sig^0_k\le ( p-1)/2 $ for all $n+1\le k\le \ell$.
Since $Z$ is a decreasing sequence in $\zz_p$,
it follows that $\delta_k=\sig^0_k=0$ for all $k\ge \ell+1$,
and $\sig^0_k=0$ for all $n+1\le k\le \ell$.
Hence,
\begin{equation}
\sum\zeta^1 Q =Z_n - \sum\delta Z=Z_n.\label{eq:QnZn}
\end{equation}

Let us show that $\zeta^1_n=1$.
Write $Q_n=\sum \al Z=\al_n Z_n + \sum \al^0 Z$
where $\al\in \cE$ with $\ord_p(\al)=n$ and $\al^0:=\resab{(n,\infty)}(\al)$.
Recall $m:=\ord_p(Z_n)$, then $Q_n=\sum \al Z\equiv \al_n Z_n \mod p^{m+1} $, and hence,
\begin{gather}
Z_n= \zeta^1_n Q_n +\sum \resab{(n,\infty)}(\zeta^1) Q
\implies Z_n \equiv \zeta^1_n Q_n \mod p^{m+1} \notag\\
\implies
Z_n \equiv \zeta^1_n \al_n Z_n \mod p^{m+1}
\implies
1 \equiv\zeta^1_n\al_n \mod p\notag\\
\implies
\zeta^1_n\al_n = ps + 1\text{ for $s\ge 0$}.
\label{eq:root-p}
\end{gather}
On the other hand, $\zeta^1_n \al_n < (\sqrt p)^2=p$,
and the equation (\ref{eq:root-p}) implies that $s=0$ and $\zeta^1_n=1$.
\end{proof}

Let us prove Theorem \ref{thm:p-adic}, Part 2.
Given a positive integer $n$, by Lemma \ref{lem:QnZn},
we have $Z_n=\sum \zeta Q$ and $Q_n = \sum \xi Z$
where $\ord_p(\zeta)=\ord_p(\xi)=n$, $\zeta_n=\xi_n=1$, and $\zeta$ and $\xi$ are s at index $s$ and $t$, respectively. So,
\begin{gather}
Z_n = \sum\zeta Q = Q_n +\sum_{k=n+1}^s \zeta_k Q_k
\implies
Z_n=
\left(Z_n + \sum_{k=n+1}^t \xi_k Z_k \right)+ \sum_{k=n+1}^s \zeta_k Q_k\notag \\
\implies
0= \sum_{k=n+1}^t \xi_k Z_k + \sum_{k=n+1}^s \zeta_k Q_k.
\label{eq:remainder}
\end{gather}
Let us show that $\xi_k=\zeta_k=0$ for all $k\ge n+1$.
Suppose that there is a smallest index $\ell \ge n+1$ such that
$\xi_\ell\ne 0$ or $\zeta_\ell\ne 0$; without loss of generality,
assume that $\ell=n+1$ and $\xi_{n+1}\ne 0$.
By Lemma \ref{lem:QnZn}, there is an \Eb\ $\theta$ at index $u$
with $\ord_p(\theta)=n+1$ and $\theta_{n+1}=1$ such that
$Z_{n+1} = \sum \theta Q$. Thus,
\begin{gather*}
0= \xi_{n+1} Z_{n+1}+ \sum_{k=n+2}^t \xi_k Z_k
+\zeta_{n+1} Q_{n+1}+ \sum_{k=n+2}^s \zeta _k Q_k\\
=
\xi_{n+1}\left( Q_{n+1}+ \sum_{k=n+2}^u \theta_k Q_k \right)+ \sum_{k=n+2}^t \xi_k Z_k
+\zeta_{n+1} Q_{n+1}+ \sum_{k=n+2}^s \zeta _k Q_k\\
\intertext{If $r:=\ord_p(Q_{n+1})$, then it follows}
0\equiv (\xi_{n+1} +\zeta_{n+1} )Q_{n+1} \mod {p^{r+1}}.
\end{gather*}
Since $\xi_{n+1} +\zeta_{n+1}\le p-1$, we find
$\xi_{n+1}=\zeta_{n+1}=0$.
Therefore, $Q_n=Z_n$, and this concludes the proof of Theorem \ref{thm:p-adic}, Part 2.

\subsection{The \Lunique}\label{sec:proof-unique-property}

In this section we prove Theorem \ref{thm:unique}.
Recall that a \zml\ $L=(e_1,\dots,e_N)$ for $N\ge 2$ considered in the theorem is an arbitrary list of positive integers  $e_k $ for $k=1,\dots,N$,
and let $\cL$ denote the \zec\ collection determined by $L$.

Let $\theta$ be a    \cf\ $\theta\in \cL$.  
The \cf\ is called a {\it proper $L$-block at index $n$}
if $\theta$ is a proper  \LLb\ at index $n$ with support $[s,n]$ such that
$n-N+1\le s\le n$.
If   $\theta$ has support $[n-N,n]$  where $n-N\ge 1$  and $\theta_{n-N}=0$,
we call $\theta$ the {\it maximal \Lb\ at index $n$}, and 
if $\theta$ is a maximal \LLb\ at index $n$ where $n\le N$,
it is also called a maximal \Lb\ at index $n$.
{\it The \Lb\ interval} of an \Lb\ $\theta $ is defined to be the interval
of indices $[s,n]$ if $\theta$ is a proper \LLb\ with support $[s,n]$.
{\it The \Lb\ interval of the maximal \Lb\ at index $n$} is defined to be $[n-N+1,n]$ if $n>N$, and $[1,n]$ if $n\le N$.
For example, if $L=(2,3,2)$, then $\theta^0 = \beta^4 + 3\beta^5 + 2\beta^6$
is the maximal \Lb\ at index $6$ with \Lb\ interval $[4,6]$
while $\theta^0$ is not a maximal \LLb\ at index $6$, and the support of the proper \LLb\ $\theta^0 =0\cdot\beta^3 + \beta^4 + 3\beta^5 + 2\beta^6$ is $[3,6]$.
If $\theta$ is a proper \Lb,
then the \Lb\ interval of $\theta$ coincides with the support
of the proper \LLb\ $\theta$.

For each \cf\ $\mu\in\cL$, there are unique \Lb s $\theta^m$ for $m\ge 1$
with \Lb\ interval $[s_m,n_m]$
such that $\mu= \sum_{m=1}^M \theta^m$, $s_1=1$, and
$n_m+1=s_{m+1}$.
We call it {\it the \Lb\ decomposition} of $\mu$.
In addition, if we use non-zero \Lb s only, it is called {\it the non-zero \Lb\ decomposition} of $\mu$.
If $\mu= \sum_{m=1}^M \theta^m$ is the non-zero \Lb\ decomposition
and $[i_m,n_m]$ is the \Lb\ intervals of $\theta^m$
such that $n_m+1=i_{m+1}$ for all $1\le m\le M-1$,
then the decomposition is said to {\it
have no gaps between the \Lb\ intervals.}
The sum of any \Lb s with disjoint \Lb\ intervals
is a member of $\cL$, and it is called
{\it a sum of disjoint \Lb s}.

\begin{lemma}\label{lem:cascade}
Let $Q$ be a sequence in $\nat$ given by the recurrence (\ref{eq:L-recursion}).
\begin{enumerate}

\item
Let $\theta$ be an \Lb\ with \Lb\ interval $[a,\ell]$ where $\ell\ge N$, and
let $\bar\theta$ be the maximal \Lb\ at index $\ell$.
Then, $\sum\theta Q \le \sum\bar\theta Q$, and
$Q_a +\sum \theta Q \le Q_{\ell+1}$.
In particular,
$\sum \theta Q <Q_{\ell+1}$.

\item Let $\mu^1$ be a non-zero \cf\ in $\cL$ with the non-zero \Lb\ decomposition
$\mu^1= \sum_{m=1}^M \theta^m$
where $\ord(\theta^1)\ge N$.
Let $[a,\ell]$ be the \Lb\ interval of $\theta^1$, and $n:=\ord(\theta^M)$.
\begin{enumerate}
\item
If all the \Lb s are maximal and there are no gaps between them, then
$Q_a + \sum\mu^1 Q =Q_{n+1}$.

\item
If there are no gaps between the \Lb\ intervals, and there is a smallest non-zero non-maximal \Lb\ with \Lb\ interval $[b,j]$, then
$Q_a + \sum\mu^1 Q = Q_b + \sum\resab{[b,n]}(\mu^1) Q<Q_{n+1}$.

\item If there is a smallest index $ c$ such that $\ell<c<n$, and $c$ is not contained
in any of the \Lb\ intervals,
then $Q_a + \sum\mu^1 Q \le Q_c+ \sum\resab{(c,n]}(\mu^1) Q<Q_{n+1}$.

\item
For all cases, $Q_a + \sum\mu^1 Q \le Q_{n+1}$, and hence,
$ \sum\mu^1 Q<Q_{n+1}$.
\end{enumerate}

\item Let $\mu^0$ be a \cf\ in $\cL$ with the \Lb\ decomposition
$\mu^0= \sum_{m=1}^M \theta^m$.
If $n:=\ord(\theta^M)< N$,
then $\sum\mu^0 Q < Q_{2N}$.
\end{enumerate}

\end{lemma}
\begin{proof}

Let us prove Part 1.
Since $\ell\ge N$, by the linear recurrence,
$Q_{\ell+1}= \sum_{k=1}^{N } e_k Q_{\ell-k+1}
=\sum\bar\theta Q + Q_{\ell-N+1}$
\begin{gather*}
\begin{aligned}
\implies\quad
Q_{\ell+1} -\sum \theta Q- Q_a
	&= \sum_{k=1}^{N } (e_k - \theta_{\ell-k+1}) Q_{\ell-k+1}- Q_a\\
&=(e_{\ell-a+1} - \theta_{a}) Q_{a}
+\sum_{k=\ell-a+2}^{N} e_k Q_{\ell-k+1}- Q_a.
\end{aligned}\\
\intertext{Since $e_{\ell+1-a}-\theta_a\ge 1$,}
Q_{\ell+1} -\sum \theta Q- Q_a
\ge Q_{a}
+\sum_{k=\ell-a+2}^{N} e_k Q_{\ell-k+1}- Q_a
\ge \sum_{k=\ell-a+2}^{N} e_k Q_{\ell-k+1}\ge 0,\\
\intertext{This proves
$Q_{\ell+1}\ge \sum \theta Q+ Q_a>\sum \theta Q$.
Also, $
Q_{\ell+1}=Q_{\ell-N+1} + \sum\bar\theta Q$ implies}
\sum_{k=\ell-a+2}^{N} e_k Q_{\ell-k+1}\le
Q_{\ell+1} -\sum \theta Q- Q_a=
\left( Q_{\ell-N+1} + \sum\bar\theta Q \right) -\sum \theta Q- Q_a
\\
\implies\quad
\sum_{k=\ell-a+2}^{N} e_k Q_{\ell-k+1} + Q_a - Q_{\ell-N+1}
\le
\sum\bar\theta Q -\sum \theta Q
\end{gather*}
If $a=\ell-N+1$, then
$\sum_{k=\ell-a+2}^{N} e_k Q_{\ell-k+1}=0$, and
$\sum\bar\theta Q -\sum \theta Q \ge 0$.
If $a>\ell-N+1$, then $\sum_{k=\ell-a+2}^{N} e_k Q_{\ell-k+1}
\ge e_N Q_{\ell-N+1}$, so we have $\sum\bar\theta Q -\sum \theta Q \ge Q_a>0$ as well.

We shall use the mathematical induction on $M$ to prove Part 2.
Since Part 2 (c) makes sense only for $M\ge 2$,
the initial cases of the induction are $M=1,2$.
Since the proofs of $M=1,2$ are very similar to the proof of
the main induction step with arbitrary $M$, we leave it to the reader.
Assume that Part 2 is true for some $M\ge 2$.
Let
$\mu^1= \sum_{m=1}^{M+1} \theta^m$ be a
non-zero \Lb\ decomposition, and let
$[i_m,n_m]$ be the \Lb\ interval of $\theta^m$.
Suppose there are no gaps between the \Lb\ intervals.
If the \Lb\ intervals are all maximal,
then by the linear recurrence, $Q_{i_1} + \sum_{m=1}^{M+1} \theta^m Q
= Q_{i_2}+ \sum_{m=2}^{M+1} \theta^mQ$.
By the induction hypothesis for Part 2 (a),
$
Q_{i_2}+ \sum_{m=2}^{M+1} \theta^m Q=Q_{n_{M+1}+1}$.
So, Part 2 (a) for $M+1$ is proved.
If there is a smallest index $t\ge 1$ such that $\theta^t$ is not maximal,
then
$Q_{i_1} + \sum_{m=1}^{M+1} \theta^m Q
= Q_{i_t}+ \sum_{m=t}^{M+1} \theta^mQ$ by
the induction hypothesis for Part 2 (a).
Since $n_t \ge N$, we have
$ Q_{i_t}+\sum \theta^t Q + \sum_{m=t+1}^{M+1} \theta^mQ
\le Q_{i_{t+1}} + \sum_{m=t+1}^{M+1} \theta^mQ$
by Part 1.
Thus, by the induction hypothesis for Part 2 (d),
$Q_{i_1} + \sum_{m=1}^{M+1} \theta^m Q
<Q_{i_{t+1}} + \sum_{m=t+1}^{M+1} \theta^mQ
\le Q_{n_{M+1}+1}$.
This proves Part 2 (b) for $M+1$.
Suppose that there is a smallest gap $c$ between the \Lb\ intervals.
Then, $n_k+1=c<i_{k+1}$ for some $1\le k \le M$, and
by the induction hypothesis for Part 2 (d),
$$
Q_{i_1} + \sum_{m=1}^{M+1} \theta^m Q
=Q_{i_1} + \sum_{m=1}^{k} \theta^m Q
+ \sum_{m=k+1}^{M+1} \theta^m Q
\le Q_c +\sum_{m=k+1}^{M+1} \theta^m Q. $$
Since $c\ge N$, by the linear recurrence,
$Q_c < Q_{i_{k+1}}$, so
$Q_c +\sum_{m=k+1}^{M+1} \theta^m Q <
Q_{i_{k+1}} +\sum_{m=k+1}^{M+1} \theta^m Q
\le Q_{n_{M+1}+1}$ where
the induction hypothesis for Part 2 (d) is used for
the last inequality.
This proves Part 2 (b) for $M+1$.
Since any non-zero \Lb\ decomposition satisfies
one of the three possibilities for the gaps between \Lb\ intervals
described in Part 2 (a,b,c), we proved Part 2 (d) for $M+1$.

Let us prove Part 3.
Notice that the linear recurrence implies that $Q_k < Q_{k+N}$ for all $k\ge 1$, and
hence, $\sum \mu^0 Q = \sum_{k=1}^{N-1} \mu^0_k Q_k
<\sum_{k=1}^{N-1} \mu^0_k Q_{k+N}$.
Let $\mu^1 = \sum_{k=N+1}^{2N-1}\mu^0_{k-N}\beta^k$.
Then, $\sum \mu^1 Q = \sum_{k=1}^{N-1} \mu^0_k Q_{k+N}$,  
$n:=\ord(\mu^1)\le 2N-1$, and 
 $\mu^1$ has a decomposition into $L$-blocks as described in Part 2.
Hence, by Part 2 (d),
we have 
$\mu_0<\sum_{k=1}^{N-1} \mu^0_k Q_{k+N} = \sum \mu^1 Q < Q_{n+1} \le Q_{2N}$.

\end{proof}

\newcommand{\muo}{\mu^1}
\newcommand{\muz}{\mu^0}
\newcommand{\sigo}{\sig^1}
\newcommand{\sigz}{\sig^0}
Let us prove Theorem \ref{thm:unique}.
In fact, we prove a more specific version of the theorem, and it is stated below.
\begin{theorem}\label{thm:unique-2}
Let $L$ be a \zml\ $(e_1,\dots,e_N)$ such that $e_k\ge 1$ for all $1\le k\le N$,
and let $Q$ be a sequence in $\nat$.
Let $\sig$ and $\mu$ be two non-zero \cf s in $\cL$ such that
$\sum \sig Q = \sum \mu Q$.
\begin{enumerate}
\item
If
$\sig_k =\mu_k$ for all $k> 3N$,
then there are \Lz\ \cf s $\al$ and $\gamma$ of order $\le 4N-1$ and $\rho\in\cL$ such that
$
\sig =\gamma + \rho$ and $
\mu = \al + \rho$, and
the \Lb\ intervals of non-zero \Lb s in $\rho$ are
contained in $(M,\infty)$ where $M=\max\set{\ord(\gamma),\ord(\al)}$.

\item
If there is a largest integer $K> 3N$ such that
$\sig_K >\mu_K$,
then there are \cf s $\al$ and $\gamma$ in $\cL$, and an index $a$ such that
$\ord(\al)\le 3N-2$, $\ord(\gamma)\le 2N-1$,
$\max\set{\ord(\al)+1,2N}\le a\le 3N$, and
\begin{align*}
\sig &=\gamma + (c+1)\beta^{1+\ell + tN} + \rho\\
\mu &= \al +\theta+ \bar\theta+ c\beta^{1+\ell + tN} + \rho
\end{align*}
where $\theta$ is a non-zero \Lb\ with \Lb\ interval $[a,\ell]$,
$\bar\theta=\resab{(\ell,\ell+tN]}(\hbeta^{1+\ell + tN})$, and $\rho$ is a \cf\ supported on $(1+\ell+tN,\infty)$ such that
$(c+1)\beta^{1+\ell + tN} + \rho$ is a sum of disjoint \Lb s for some
integer $c\ge 0$.
Moreover, the equality $\sum\sig Q = \sum\mu Q$ is
reduced under
the $L$-linear recurrence to the equation $\sum( \al+\theta) Q =
\sum \gamma Q + Q_{\ell+1}$.

\end{enumerate}

\end{theorem}
For example, if $L=(11,3)$ and $(Q_1,Q_2)=(19,3)$,
then $2Q_3 = 3Q_2 + 9Q_1$, and
\begin{align*}
(11Q_4+Q_3)+2Q_3 &= (11Q_4+Q_3)+3Q_2 + 9Q_1\\
\implies
Q_5 & = 11Q_4+Q_3 +3Q_2 + 9Q_1\\
\implies
Q_7 & =11Q_6+2Q_5+ 11Q_4+Q_3 +3Q_2 + 9Q_1.
\end{align*}
Then, the last expression
fits into Part 2 of Theorem \ref{thm:unique-2}
with $\gamma=\rho=0$, $\theta = \beta^3 + 11\beta^4$,
$\bar\theta = 2\beta^5 + 11\beta^6$,
and
$\al=9\beta^1 +3\beta^2$.

\begin{proof}

Let $Q$ be a sequence in $\nat$ such that $\sum\mu Q = \sum\sig Q$
where $\mu$ and $\sig$ are distinct non-zero \cf s in $\cL$,
and
$\mu \aord \sig$.
Let us prove Part 1.
Suppose that $\sig_k = \mu_k$ for all $k>3N$.
Then,
the \Lb\ intervals of (non-zero)
\Lb s of order $\ge 4N$ are contained in
$[3N+1,\infty)$, and hence,
the set of \Lb s of order $\ge 4N$ in $\sig$
is equal to the set of \Lb s of order $\ge 4N$ in $\mu$ .
Thus, after cancelling the \Lb s of order $\ge 4N$,
the equality reduces to $\sum \sig' Q = \sum \mu' Q$
where $\sig'$ and $\mu'$ are distinct members of $\cL$, and their orders are $\le 4N-1$. So, the theorem is proved for this case.

Let us prove Part 2.
Suppose that there is a largest index $K>3N$ such that $\sig_K > \mu_K$.
Then, let us claim that there is an \Lb\ in $\mu$
with \Lb\ interval $[K,s]$.
Notice that $\sig_k = \mu_k$ for all $k>K$, and
let $\theta^1$ and $\theta^2$ be the smallest \Lb s of $\sig$ and $\mu$ at index $>K$, respectively, including zero \Lb s.
Lemma \ref{lem:E-unique}, Part 1 is valid for \Lb\ decompositions as well,
and since $\sig_k=\mu_k$ for $k>K$,
the two blocks
$\theta^1$ and $\theta^2$ must have the same index $s$
.
Notice that the choice of $\theta^1$ implies that
the \Lb\ interval of $\theta^1$ is either $[K+1,s]$ or $[t,s]$ for $t\le K$,
and we have the inequality $\mu_K< \sig_K \le e_j$ for some $1\le j\le N$.
If it is $[K+1,s]$, then it follows that $\mu_K\beta^K$ where
$\mu_K<e_1$ is an
\Lb\ of $\mu$ with \Lb\ interval $[K,K]$, and
if it is $[t,s]$ for $t\le K$, then it follows that $\theta^2$ is an \Lb\ of $\mu$ with \Lb\ interval $[K,s]$.
It proves the claim, and let $\mu^*:=\resab{[K,\infty)}(\mu)$.
Then, $\mu^*$ is a sum of disjoint \Lb s,
$\sig \equiv \mu^* \reS (K,\infty)$, and $\sig_K \ge \mu^*_K+1$.

Let $\muz$ be the sum of \Lb s in $\mu$ whose \Lb\ intervals are
contained in $[1,N)$, and let
$\mu^1:=\mu-\mu^*-\muz$.
Then, $\mu^1\in\cL$ since it is a sum of disjoint \Lb s, and
$\mu = \mu^* + \mu^1 + \mu^0$.
Let $\sigz$ be the sum of \Lb s in $\sig$ whose \Lb\ intervals are
contained in $[1,N)$, and let
$\sigo:=\sig-\sigz$.
Then, $\sigo\in\cL$, and $\sig = \sigo + \sigz$.
Also,
$\sum \sigo Q \ge \sum \mu^* Q + Q_K$.

By Lemma \ref{lem:cascade}, Part 3,
we have $ \sum \muz Q - \sum\sigz Q \le \sum \muz Q<Q_{2N }$, and hence,
\begin{equation}\label{eq:cascade}
\sum \sigo Q = \sum \mu^* Q+ \sum \muo Q + \sum \muz Q - \sum \sigz Q
<\sum \mu^* Q+\sum \muo Q + Q_{2N } .
\end{equation}%
\newcommand{\muoo}{\mu^{11}}
\newcommand{\muot}{\mu^{10}}
Let us further decompose $\mu^1$.
Let $\muoo$ be
the sum of all \Lb s in $\muo$
whose \Lb\ intervals are contained in $[2N,\infty)$, and
let $\muot:=\muo-\muoo \in \cL$.
Then, $\mu^1 = \muoo + \muot$.
Since the number of indices in the \Lb\ interval of an \Lb\ is $\le N$,
we have $\ord(\muot)\le 3N-2$;
otherwise, the \Lb\ interval of a largest non-zero \Lb\ in $\muot$ is
contained in $[2N,\infty)$.

Let us prove that $\muoo$ is non-zero.
Assume that $\muoo=0$, and let $m:=\ord(\muot)\le 3N-2$.
By Lemma \ref{lem:cascade}, Part 2 (d),
$\sum\muot Q <Q_{m+1}\le Q_{3N-1}$,
and hence,
the inequality (\ref{eq:cascade}) implies
$\sum \sigo Q <\sum \mu^* Q+ \sum \muot Q + Q_{2N }
< \sum \mu^* Q+ Q_{3N}+ Q_{2N }$.
Since $Q_{3N}+Q_{2N}$ is the sum of distinct \Lb s, Lemma \ref{lem:cascade}, Part 2 (d) implies that $Q_{3N}+Q_{2N}<Q_{3N+1}\le Q_K$, and it yields
$\sum \sigo Q <\sum \mu^* Q + Q_K\le \sum \sigo Q$.
Since it implies $\sum \sigo Q < \sum\sigo Q$, we prove that $\muoo$ is non-zero.

Let $\theta$ be the smallest non-zero \Lb\ in $\muoo$,
let $[a,\ell]$ be its \Lb\ interval, and let $n:=\ord(\muoo)$.
Let us prove that $\muoo =
\theta+ \resab{(\ell,K)}(\hbeta^K)$.
Suppose that there is a smallest zero \Lb\ in $\muoo$ at index $c$ such that $a\le \ell < c<n$.
Then, by Lemma \ref{lem:cascade}, Part 2 (c),
\begin{align}
\sum \sigo Q &<\sum \mu^* Q+\sum \muoo Q +\sum \muot Q + Q_{2N }
\le\sum \mu^* Q+\sum \muoo Q + Q_{a }+\sum \muot Q\notag \\
& \le \sum \mu^* Q+\sum\resab{(c,n]}( \mu^{11})Q+Q_c +\sum \muot Q .
\label{eq:mu-decomp}
\end{align}
Since $e_2\ge 1$, the \cf\ $\beta^c$
makes an \Lb\ with \Lb\ interval $[c,c]$ or $[c-1,c]$.
Notice that $\theta+\muot$ is the sum of disjoint \Lb s,
and $c-1\ge a$.
Thus, the \cf\ of the last three terms of (\ref{eq:mu-decomp}) is
the sum of
disjoint \Lb s.
By Lemma \ref{lem:cascade},
Part 2 (d), we have
$\sum \sigo Q <\sum \mu^* Q + Q_{n+1} \le \sum \mu^* Q + Q_K\le \sum\sigo Q$.
Thus, the non-zero \Lb s in $\muoo$ have no gaps between their \Lb\ intervals.
In fact, even if some $e_j$ are zeros, $\theta$ being a non-zero \Lb\ implies
that $\beta^c + \muot$ is the sum of disjoint \Lb s, but
we do require $e_k\ge 1$ for the next part.

Suppose that $\muoo$ has a smallest non-zero \Lb\ $\theta^0$ with \Lb\
interval $[b,s]$ such that $\theta^0$ is not maximal, and $a\le \ell <b$.
Then, by Lemma \ref{lem:cascade}, Part 2 (b),
\begin{align*}
\sum \sigo Q &<\sum \mu^* Q+\sum \muoo Q +\sum \muot Q + Q_{2N }
\le\sum \mu^* Q+\sum \muoo Q + Q_{a }+\sum \muot Q \\
& \le \sum \mu^* Q+\sum \resab{[b,n]}(\muoo)Q + Q_b +\sum \muot Q\\
&=\sum \mu^* Q+\sum \resab{(s,n]}(\muoo)Q
+\sum(\theta^0 + \beta^b)Q +\sum \muot Q.
\end{align*}
There is a non-negative integer $v$ such that
$\theta^0=\sum_{k=1}^{s-b} e_k \beta^{s-k+1} + v \beta^b$
and $v<e_{s-b+1}$.
If $s-b+1=N$,
then $v<e_N-1$ since the \Lb\ $\theta^0$ is non-maximal, and hence, $\theta^0+\beta^b$ is an \Lb\ with \Lb\ interval
$[b,s]$.
Otherwise, the \Lb\ interval is either $[b,s]$ or $[b-1,s]$ since
$e_{s-b+2}\ge 1$.
Since $b-1\ge a$, the last three terms of the last expression is the
sum of disjoint \Lb s.
By Lemma \ref{lem:cascade},
Part 2 (d), we have
$\sum \sigo Q <\sum \mu^* Q + Q_{n+1} \le \sum \mu^* Q + Q_K\le \sum\sigo Q$.

Thus, $\muoo = \theta+\resab{(\ell,n]}(\hbeta^{n+1})$.
Let us show that $n=K-1$.
If $n<K-1$, then by Lemma \ref{lem:cascade}, Part 2 (d),
\begin{gather*}
\sum \sigo Q < \sum \mu^* Q+\sum \muoo Q +\sum \muot Q + Q_{2N }
\le\sum \mu^* Q+\sum \muoo Q + Q_{a }+\sum \muot Q \\
\le
\sum \mu^* Q+Q_{n+1}+\sum \muot Q
< \sum \mu^* Q+Q_{n+2}
\le \sum \mu^* Q+Q_K \le \sum\sig^1 Q.
\end{gather*}
This concludes the proof of
$\muoo=\theta + \resab{(\ell,K)}(\hbeta^{K})$.
Moreover, if $\mu_K+1<\sig_K$, then
$\sum \sigo Q <\sum \mu^* Q+\sum \muoo Q + Q_{a}+\sum \muot Q
\le\sum \mu^* Q+ Q_K+\sum \muot Q<\sum \mu^* Q+ Q_K+ Q_{m+1}
<\sum \mu^* Q+ Q_K+ Q_{K}$
by Lemma \ref{lem:cascade}, Part 2 (d) where $m=\ord(\muot)\le 3N-2$.
Then,
$\sum \sigo Q <\sum \mu^* Q+ 2Q_K\le
\sum\sigo Q$.
Thus, we have $\mu_K+1=\sig_K$ as well.

Let us also show that $a\le 3N$.
If $a\ge 3N+1$, then
\begin{gather*}
\sum \sigo Q < \sum \mu^* Q+\sum \muoo Q + Q_{2N}+\sum \muot Q\\
<\sum \mu^* Q+\sum \muoo Q + Q_{3N}+\sum \muot Q.
\end{gather*}
The \cf\ $\beta^{3N}$ or $\beta^{3N}+0\cdot \beta^{3N-1}$ makes an
\Lb\ since $e_2\ge 1$,
and since $\ord(\muot)\le 3N-2$ and $a\ge 3N+1$, the \cf\
$\mu^* + \muoo +\beta^{3N}+ \muot $ is the sum of disjoint
\Lb s.
By Lemma \ref{lem:cascade}, Part 2 (d),
the previous inequality implies
$ \sum \sigo Q < \sum \mu^* + Q_K\le \sum \sigo Q$.

Let us claim that $\sig_k=0$ for all $2N\le k<K$.
Suppose that $\sig_r >0$ for some $2N\le r<K$.
Then, $\sigo_K = \mu^*_K+1$ and
$Q_K\ge \sum\muoo Q + Q_a$ imply
\begin{align*}
\sum \sigo Q -\sum\mu^* Q- \sum\muo Q
	&\ge Q_K + Q_r - \sum\muoo Q - \sum \muot Q \ge Q_r + Q_a- \sum \muot Q .
\end{align*}
If $ \muot $ is non-zero, then
by Lemma \ref{lem:cascade}, Part 2 (d),
$\sum \muot Q < Q_{m+1}$ where $N\le m=\ord(\muot)<a$,
and hence, $ \sum \muot Q < Q_{m+1}\le Q_a$ since $Q_{k}$ for $k\ge N$
is an increasing sequence.
Thus,
\begin{align*}
\sum \sigo Q -\sum\mu^* Q - \sum\muo Q > Q_r \ge Q_{2N },
\end{align*}
and this holds for $\muot=0$ as well.
On the other hand, $
\sum \sigo Q -\sum\mu^* Q- \sum\muo Q =\sum\muz Q - \sum \sigz Q
<Q_{2N }$,
which contradicts the above inequality.
Therefore,
$\sig_k = 0$ for all $2N\le k <K$.

Recall that $\hat e_k:=e_k$ if $1\le k\le N-1$, and
$\hat e_{N-1}:=e_N-1$.
Then, the linear recurrence implies
\begin{gather*}
Q_K =\sum_{k=1}^N \hat e_k Q_{K-k}
+\sum_{k=1}^N \hat e_k Q_{K-N-k}+\cdots
+ \sum_{k=1}^N e_k Q_{\ell-k+1} \\
\implies
\sum\muz Q - \sum \sigz Q
= \sum \sigo Q -\sum\mu^* Q - \sum\muoo Q - \sum\muot Q \\
\phantom{\sum\muz Q - \sum \sigz Q }
\begin{aligned}
&
=\sum \resab{[1,2N)}(\sig^1)Q + Q_K
- \sum \theta Q
-\sum \resab{(\ell,K)}(\hbeta^K)Q - \sum \muot Q\\
&
=\sum \resab{[1,2N)}(\sig^1)Q +
\sum_{k=1}^N e_k Q_{\ell-k+1}
-\sum \theta Q - \sum \muot Q
\end{aligned}\\
\implies
\sum\muz Q + \sum \muot Q + \sum \theta Q
= \sum \sigz Q +\sum \resab{[1,2N)}(\sig^1)Q + Q_{\ell+1}\\
\phantom{\implies\sum\muz Q + \sum \muot Q + \sum \theta Q}
= \sum \resab{[1,2N)}(\sig)Q + Q_{\ell+1}.
\end{gather*}
Notice that $Q_{\ell+1}$ or $0\cdot Q_\ell + Q_{\ell+1}$ makes
an \Lb.
Since $2N+1\le a+1\le \ell+1 \le a+N\le 4N $,
the \cf\ $\resab{[1,2N)}(\sig)+Q_{\ell+1}$
is a sum of disjoint \Lb s, and it has order $\le 4N$.
The \cf\ of the LHS is a sum of disjoint \Lb s, which has order $\le 4N-1$.
This concludes the proof of Part 2.

\end{proof}

Theorem \ref{thm:unique-2} proves the if-part of Theorem \ref{thm:unique}, and
the only-if-part is trivial.

\section{Future Work}

Recall Theorem \ref{thm:daykin}, the full converse of \zec's Theorem, and that
the full converse fails for $\nat$ under
the \zec\ condition $\cE$ defined in Example \ref{exm:converse-fail}.
We may ask ourselves
whether the full converse holds for all $\cE$-subsets of $\nat$ if it holds for $\nat$, or which additional conditions on \zec\ collections will guarantee the full converse.
For example, if $\cL$ is the \zec\ collection defined by the \zml\ $L=(1,1)$
and $Q$ is a sequence given by $Q_k=2^{k-1}$ for $k\ge 1$,
does $X_Q^\cL$ have a unique fundamental sequence?

We proved the weak converse for $\OI$ under the \zec\ conditions determined by certain short linear recurrences considered in Theorem \ref{thm:real}, Part 2, but the weak converse for other \zec\ conditions remains to be investigated.
Even for the \zec\ conditions considered in Theorem \ref{thm:real}, Part 2, we may ask ourselves whether the weak converse holds for $\cL$-subsets of $\OI$.
For example, if
$Q$ is a sequence given by $Q_k=1/2^{k}$ for $k\ge 1$
and $L=(1,1)$,
does $X_Q^\cL$ have a unique decreasing fundamental sequence?

\end{document}